\documentclass[a4paper,10pt]{article}
\usepackage{amsmath,amssymb}
\usepackage{amsthm}
\usepackage{enumerate}
\usepackage{geometry}

\usepackage{mathtools}
\usepackage{hyperref}
\theoremstyle{definition} 
\newtheorem{theo}{Theorem}[section]
\newtheorem{defi}[theo]{Definition}
\newtheorem{prop}[theo]{Proposition}
\newtheorem{lemm}[theo]{Lemma}
\newtheorem{cor}[theo]{Corollary}
\newtheorem{rem}[theo]{Remark}

\newtheorem*{claim*}{Claim}
\usepackage[all]{xy}
\newcommand{\N}{\mathbb{N}}
\newcommand{\Z}{\mathbb{Z}}
\newcommand{\Q}{\mathbb{Q}}

\newcommand{\diff}{\frac{d}{d \lambda}}
\newcommand{\dtwo}{\frac{d^2}{d \lambda^2}}
\newcommand{\la}{\lambda}
\newcommand{\q}[1]{\left[ #1\right]_q}
\newcommand{\Rla}{R[[\la]]}
\newcommand{\Rcompla}{R\left\langle \la \right\rangle }
\newcommand{\Fil}{\text{Fil}}
\newcommand{\modphib}{\mod \q{p^{s+1}} \varphi^{s+1} ( B^{(s)}(m) )}
\newcommand{\Qlalog}{Q((\la))\left[ \log_q \la \right]}
\newcommand{\MICHB}{\left(H_{0}, \nabla ,\varphi_{H_{0}}, \Fil^{\bullet} H_{0} \right)}

\title{$q$-deformation with ($\varphi, \Gamma$) structure of the de Rham cohomology of the Legendre family of elliptic curves}
\author{Ryotaro Shirai}
\date{\today}

\begin{document}
	\maketitle
	\begin{abstract}
		In the late '60s, B. Dwork studied a Frobenius structure compatible with the classical hypergeometric differential equation with parameters $\left(\frac{1}{2},\frac{1}{2} ; 1 \right)$ by analyzing behavior of solutions of the differential equation under Frobenius transformation.
		Recently, P. Scholze conjectured the existence of $q$-de Rham cohomology groups for any $\Z$-scheme. 
		In this paper, we give a Frobenius structure compatible with the $q$-hypergeometric differential equation with parameters $(q^{\frac12},q^{\frac12};q)$ by showing a $q$-analogue of some results of Dwork. 
		This construction gives a $q$-deformation with $(\varphi,\Gamma)$-structure over $\Z_p[[q-1]][[\la]]$ of the de Rham cohomology of the $p$-adic Legendre family of elliptic curves which has Frobenius structure and connection.
	\end{abstract}
		
	\section{Introduction}
	\label{introduction}
	Let $p$ be an odd prime number.
	Let	$h(\la) = \sum_{i=0}^{\frac{p-1}{2}} {\binom{\frac{p-1}{2}}{i} }^2 \la^i$ be the Hasse polynomial.
	Let $B=\Z_p \left\langle \la, \frac{1}{\la (1-\la)h(\la)} \right\rangle$ be the $p$-adic completion of the ring $\Z_p\left[ \la,\frac{1}{\la (1-\la)h(\la)} \right] $.
	We consider the $p$-adic Legendre family of elliptic curves
	\[
	E=\text{Proj} \left( B[X,Y,Z]/(Y^2Z -X(X-Z)(X-\la Z))\right) \to \text{Spec}(B).
	\]
	Let $\overline{\Q_p}$ be an algebraic closure of $\Q_p$.
	Let $\overline{\Z_p}$ be the integral closure of $\Z_p$ in $\overline{\Q_p}$.
	Let $\mathfrak{p}$ be the maximal ideal of $\overline{\Z_p}$.
	For every value $\mu$ $\left(\mu \in \overline{\Z_p}, \mu(1-\mu) h(\mu)\neq 0 \mod\mathfrak{p} \right)$ of $\la$, the fiber above $\mu$ is an elliptic curve with good ordinary reduction, denoted by $E_{\mu}$.
	The relative curve $E$ over $B$ with the divisor at infinity deleted is written as $\text{Spec}\left( B[x,y]/(y^2-x(x-1)(x-\la))\right)$,	where $(x,y)=\left(\frac{X}{Z},\frac{Y}{Z}\right)$.
	Then the de Rham cohomology $ H^1_{\text{dR}} \coloneqq H^1_{\text{dR}}(E/B)$ is a free $B$-module of rank $2$ and $\text{Fil}^1H^1_{\text{dR}}=\Gamma(E,\Omega_{E/B})$ is a free $B$-module of rank $1$ with basis $\omega=\frac{dx}{y}$.
	
	In~\cite[\S 6]{Dw}, Dwork defined a Frobenius structure $\varphi_{ H^1_{\text{dR}}}$ on $ H^1_{\text{dR}}$.
	Moreover, he found that there exists a unique direct summand $U$ of the $B$-module $ H^1_{\text{dR}}$ stable under $\varphi_{H^1_{\text{dR}}}$. (See~\cite{Put} and~\cite[pp. 232-233]{Fr}.)
	This $U$ is called the unit root part of $ H^1_{\text{dR}}$. Let $E_{\mathbb F_p}$ denote the reduction mod $p$ of $E$. 
	Then $H^1_{\text{dR}}$ is canonically isomorphic to the crystalline cohomology $H^1_{\text{crys}}(E_{\mathbb F_p}/B)$, and $\varphi_{H^1_{\text{dR}}}$ coincides with the Frobenius structure induced by the absolute Frobenius of $E_{\mathbb F_p}$.
	
	More precisely, we obtain $U$ as follows. 
	Let $\varphi\colon B\to B$ be the unique lifting of the absolute Frobenius satisfying $\varphi(a)=a$ $(a\in \mathbb Z_p)$ and $\varphi(\lambda)=\lambda^p$ (cf.~\S\ref{sec:basic}). 
	Then the Frobenius structure $\varphi_{H^1_{\text{dR}}}$ is realized as a $\varphi$-semilinear endomorphism of $H^1_{\text{dR}}$, which is again denoted by $\varphi_{H^1_{\text{dR}}}$ in the following.
	Let $\nabla \colon H^1_{\text{dR}} \to H^1_{\text{dR}} \otimes_B\Omega_{B}$ be the Gauss-Manin connection.
	We define the $B$-linear endomorphism $D$ of $H^1_{\text{dR}}$ by $\nabla=D\otimes d\la$.
	Then $\omega$ and $D(\omega)$ form a basis of $ H^1_{\text{dR}}$ because the Kodaira-Spencer map $\Gamma(E,\Omega_{E/B}) \subset H^1_{\text{dR}} \xrightarrow{\nabla} H^1_{\text{dR}}\otimes_B\Omega_B\to H^1_{\text{dR}}/\text{Fil}^1\otimes_B\Omega_B$ is an isomorphism. 
	(This follows from the following fact: For any field $F$ of characteristic $\neq 2$, and any $a\in F\backslash \{0,1\}$, the elliptic curve $y^2=x(x-1)(x-a-\varepsilon)$ over $F[\varepsilon]/(\varepsilon^2)$ is not constant, i.e., not isomorphic to the base change of the elliptic curve $y^2=x(x-1)(x-a)$ over $F$ under $F\hookrightarrow F[\varepsilon]/(\varepsilon^2)$.)
	By~\cite[Proposition 7.11.(ii)]{Put}, we can write $\nabla$ on $ H^1_{\text{dR}}$ explicitly as
	\begin{multline}
	\label{GMExplicitFormula}
	\nabla \left( \begin{pmatrix} \la(1-\la)\omega & -\la(1-\la)D(\omega) \end{pmatrix} \right) \\ =
	\begin{pmatrix} \la(1-\la)\omega & -\la(1-\la)D(\omega) \end{pmatrix}
	\frac{1}{\la(1-\la)} \begin{pmatrix} 1-2\la & -\frac{1}{4} \\-\la(1-\la) & 0\end{pmatrix} \otimes d\la.
	\end{multline}
	Let $C$ be a ring extension of $B$ (e.g.~$\Q_p[[\lambda]]$) which carries an extension of $\frac{d}{d\lambda}$. 
	Then the above formula implies that, for $f_1,f_2\in C$, we have $D(f_1\lambda(1-\la)\omega-f_2\lambda(1-\la)D(\omega))=0$ if and only if $f_1=\frac{df_2}{d\lambda}$ and $f_2$ satisfies the classical hypergeometric equation with parameters $\left(\frac{1}{2},\frac{1}{2} ; 1 \right)$ (\cite[Proposition 7.11.(iii)]{Put}):
	\begin{equation}
	\label{hyp-geom-eq}
	\la(1-\la) \dtwo f_2 +(1-2\la) \diff f_2 -\frac{1}{4} f_2=0.
	\end{equation}
	This differential equation has the well-known solution $f(\la)\coloneqq F(\frac12,\frac12;1;\la)=\sum_{n=0}^{\infty} \left( \prod^{n-1}_{i=0} \left( \frac{i+\frac{1}{2}}{i+1}\right)^2 \right) \la^n$$\in \mathbb Q_p[\la]]$, which converges on the open unit disk.
	
	Let $\Z_p\left\langle\la,\frac{1}{h(\la)}\right\rangle$ be the $p$-adic completion of
	$\Z_p \left[\la,\frac{1}{h(\la)} \right]$. 
	In~\cite[\S1-\S4]{Dw}, Dwork showed $\frac{\varphi(f)}{f}\in\Z_p\left\langle\la,\frac{1}{h(\la)}\right\rangle^{\times}$, from which he derived
	$\frac{1}{f}\frac{df}{d\lambda}\in \Z_p\left\langle\la,\frac{1}{h(\la)}\right\rangle$. The latter implies that
	$\overline{e}_1'=\omega$ and $\overline{e}'_2=\la(1-\la)\left(\frac{1}{f}\frac{df}{d\la} \omega-D(\omega)\right)$ form a basis of $H^1_{\text{dR}}$. 
	The unit root part $U$ is given by $U=B\overline{e}'_2$, and he further showed a formula~(\cite[(6.28)]{Dw}):
	\begin{equation}
	\label{unitrootphi}
	\varphi_{H^1_{\text{dR}}}({\overline{e}'_2})= \varepsilon\frac{f}{\varphi(f)} {\overline{e}'_2}, \qquad \varepsilon=(-1)^{\tfrac{p-1}{2} }.
	\end{equation}
	By the last claim in the previous paragraph, we see $D(f{\overline{e}'_2})=0$, which implies	$D({\overline{e}'_2})=-\frac{1}{f}\frac{df}{d\lambda}{\overline{e}'_2}$, and therefore $U$ is stable under $\nabla$.
	
	By using the explicit formula of $\nabla$ on $H^1_{\text{dR}}$ above, we see $\nabla({\overline{e}'_1}\wedge {\overline{e}'_2})=0$ and therefore $(\wedge^2H^1_{\text{dR}})^{\nabla=0}=\Z_p \left({\overline{e}'_1}\wedge {\overline{e}'_2}\right)$.
	Since the Frobenius endomorphism of $\wedge^2H^1_{\text{dR}}(E_{\mu})=H^2_{\text{dR}}(E_\mu)$ for $\mu=[a]\in \Z_p$ $(a\in \mathbb F_p\backslash\{0,1\})$ is the multiplication by $p$,
	this implies $\varphi ({\overline{e}'_1}\wedge {\overline{e}'_2})=p {\overline{e}'_1}\wedge {\overline{e}'_2}$. 
	As $\varphi_{H^1_{\text{dR}}}(\text{Fil}^1H^1_{\text{dR}}) \subset pH^1_{\text{dR}}$, we have
	\begin{equation}
	\label{dRcohphi}\varphi_{H^1_{\text{dR}}}({\overline{e}'_1})= p\varepsilon \frac{\varphi(f)}{f} {\overline{e}'_1}+ pb {\overline{e}'_2}
	\end{equation}
	for some $b\in B$.
	
	Let $B'$ be $\Z_p \left\langle \la, \frac{1}{h(\la)} \right\rangle$ equipped with the Frobenius $\varphi$ defined in the same way as $B$.
	Let $U_{B'}$ be the $B'$-submodule $B' \overline{e}_2'$ of $U$. 
	Then, by the explicit description of $\varphi_{H^1_{\text{dR}}}\vert_U$ and $\nabla\vert_U$ recalled above, we see that they induce a $\varphi$-semilinear endomorphism $\varphi_{U_{B'}}$ of $U_{B'}$, which satisfies $B'\cdot\varphi_{U_{B'}}(U_{B'})=U_{B'}$, and a connection $\nabla\colon U_{B'}\to U_{B'}\otimes_{B'}\Omega_{B'}$, where $\Omega_{B'}=B'd\lambda$. 
	We also see that $(H^1_{\text{dR}},\text{Fil}^{\bullet}, \nabla)$ has a ``$B'$-structure'' given by relative log de Rham cohomology as follows. 
	Let $T$ be $\text{Spec}(B')$ equipped with the log structure defined by the divisor $\lambda(1-\lambda)=0$. Then, by replacing $B$ with $B'$ in the definition of $E$, we obtain a log smooth extension $E'/T$ of $E/\text{Spec}(B)$. 
	Its relative log de Rham cohomology $H_{B'} \coloneqq H^1_{\text{dR}}(E'/T)$ is a free $B'$-module of rank $2$ equipped with the logarithmic Gauss-Manin connection $\nabla\colon H_{B'}\to H_{B'}\otimes_{B'}\Omega_{B',\log}$, where $\Omega_{B',\log}=B'\frac{d\lambda}{\lambda(1-\lambda)}$. 
	The pull-back by $\text{Spec}(B)\to T$ induces a $B$-linear isomorphism $H_{B'}\otimes_{B'}B\xrightarrow{\cong} H^1_{\text{dR}}$ compatible with $\nabla$, and $\text{Fil}^{\bullet}$.
	Moreover, we have $\text{Fil}^1H_{B'}=\Gamma(E',\Omega_{E'/T})=B'\omega$, and the Kodaira-Spencer map $\text{Fil}^1H_{B'}\to H_{B'}/\text{Fil}^1\otimes_{B'}\Omega_{B',\log}$ is an isomorphism. 
	Since $\nabla(\omega)=\lambda(1-\lambda) D(\omega)\otimes \frac{d\lambda}{\lambda(1-\lambda)}$, this means that both $(\omega,\la(1-\la)D(\omega))$ and $({\overline{e}'_1}, {\overline{e}'_2})$ are bases of $H_{B'}$. This implies that $U_{B'}$ is a direct factor of $H_{B'}$.
	Since $\varphi$ of $B'$ does not preserve the divisor $\la(1-\la)=0$, the comparison isomorphism with the log crystalline cohomology of $E'_{\mathbb F_p}/T$ does not give a $\varphi$-semilinear endomorphism of $H_{B'}$.
	
	Let $B''=\Z_p\left\langle \lambda, \frac{1}{(1-\la)h(\la)}\right\rangle$ be the $p$-adic completion of $\Z_p \left[\lambda,\frac{1}{(1-\la)h(\la)} \right]$, and define the Frobenius $\varphi$ of $B''$ in the same way as that of $B$. Then, since $\varphi$ of $B''$ preserves the divisor $\la=0$, the comparison isomorphism with the log crystalline cohomology induces a $\varphi$-semilinear endomorphism $\varphi_{H_{B''}}$ of $H_{B''}\coloneqq H_{B'}\otimes_{B'}B''$, which is compatible with $\varphi_{H^1_{\text{dR}}}$. As $\varphi_{H_{B''}}(\Fil^1H_{B''})\subset pH_{B''}$, we obtain $b\in B''$.
	
	In \S\ref{background}, we introduce a category $\text{MIC}_{[0,a]}(A,\varphi,\Fil^{\bullet})$ for a non-negative integer $a$ whose object is a free $A$-module of finite type $M$ with a decreasing filtration, a Frobenius endomorphism and a connection satisfying certain conditions.
	By the above construction, we obtain an object of $\text{MIC}_{[0,a]}(A,\varphi,\Fil^{\bullet})$ in each of the cases $A=B'$, $a=0$, $M=U_{B'}$ and $A=\Z_p[[\la]]$, $a=1$, $M=H_{B''} \otimes_{B''}\Z_p[[\la]]$.
	In this paper, we are interested in $q$-analogues of these objects, namely $q$-deformations involving a formal variable $q$ such that the specialization to $q=1$ recovers the original objects.
	P. Scholze made a conjecture that there exists a canonical $q$-deformation of de Rham cohomology, which is sometimes called $q$-de Rham cohomology or Aomoto-Jackson cohomology.
	(See~\cite{Sch}.)
	Especially in~\cite[\S 8]{Sch}, he asked whether there is a relation between the $q$-differential equation given by the conjectured $q$-de Rham cohomology of the Legendre family and the $q$-hypergeometric equation. Our result, which is explained below, may be regarded as positive evidences to his question.
	
	Let $S'=\Z_p[[q-1]] \left\langle \la, \frac{1}{h(\la)} \right\rangle$ be the $(p,q-1)$-adic completion of $\Z_p[[q-1]] \left[ \la, \frac{1}{h(\la)} \right]$.
	Let $R'$ be $\Z_p[[q-1]] [[\la]]$ or $S'$.
	Put $A\coloneqq R'/(q-1)$, then we can identify $A$ with $\Z_p[[\la]]$ (resp. $B'$) when $R'=\Z_p[[q-1]] [[\la]]$ (resp. $S'$).
	
	In \S\ref{sec:basic}, we give $R'$ a Frobenius structure and a $\Gamma$-action $\rho$, and then recall the definition of $q$-connections on $R'$-modules and the relation between $\rho$-semilinear $\Gamma$-actions and $q$-connections. 
	In \S\ref{background}, we introduce the category $\text{MF}^{\q{p},q-1}_{[0,a]}(R',\varphi,\Gamma)$ for a non-negative integer $a$ whose object is a free $R'$-module of finite type $M$ with a decreasing filtration, a $\varphi$-semilinear endomorphism and a $\rho$-semilinear action of $\Gamma$ satisfying certain conditions.
	Then the canonical surjection $R' \to A$
	induces a functor
	\[
	\text{MF}_{[0,a]}^{\q{p},q-1}(R',\varphi,\Gamma)
	\xrightarrow{\bmod q-1}
	\text{MIC}_{[0,a]}(A,\varphi,\Fil^{\bullet}).
	\]
	By using the equivalence of categories in~\cite[\S7]{Tsu}, we further construct a canonical right inverse of this functor when $a=0, 1$ as
	\[
	\text{MIC}_{[0,a]}(A,\varphi,\Fil^{\bullet})\xrightarrow{-\otimes_{A}R'/(q-1)^{a+1}}
	\text{MF}_{[0,a]}^{\q{p},q-1}(R'/(q-1)^{a+1},\varphi,\Gamma)
	\xleftarrow[-\otimes_{R'}R'/(q-1)^{a+1}]{\sim}
	\text{MF}_{[0,a]}^{\q{p},q-1}(R',\varphi,\Gamma)
	\]
	This applies to the objects $U_{B'}$ (for $A=B'$ and $a=0)$ and $H_{B''}\otimes_{B''}\Z_p[[\la]]$ (for $A=\Z_p[[\la]]$ and $a=1)$ mentioned above.
	
	One can ask whether there is a relationship between the canonical lifts ($q$-deformations) of $U_{B'}$ and $H_{B''}\otimes_{B''}\Z_p[[\la]]$ constructed as above and the $q$-hypergeometric differential equation~\cite{BHS} with parameters $\left(q^{\frac{1}{2}},q^{\frac{1}{2}} ; q \right)$ defined by
	\begin{equation}
	\label{q-hyp-geom-eq} 
	q\lambda(1-q\lambda) d^2_q{f} +(1-(1+\q{2}-2\q{\tfrac{1}{2}}) \la) d_q{f} -\q{\tfrac{1}{2}}^2 f=0,
	\end{equation}
	which is a $q$-analogue of the differential equation~\eqref{hyp-geom-eq}.
	We give a positive answer to this question as follows. By ``$q$-deforming'' the relations of $\nabla$ and $\varphi$ on $U_{B'}$ (resp.~$H_{B''}\otimes_{B''}\Z_p[[\la]]$) to the hypergeometric equation~\eqref{hyp-geom-eq} recalled above, we construct a Frobenius endomorphism and a $q$-connection on a free $S'$-module of rank $1$
	(resp.~a free $\Z_p[[q-1]][[\la]]$-module of rank $2$) associated with the $q$-hypergeometric differential equation~\eqref{q-hyp-geom-eq}, and show that it gives the desired canonical $q$-deformation of $U_{B'}$ (resp.~$H_{B''}\otimes_{B''}\Z_p[[\lambda]]$).
	
	In \S\ref{Main}, we state the main theorems.
	In \S\ref{qsolu_sec}, we give explicit solutions of the $q$-hypergeometric equation~\eqref{q-hyp-geom-eq}, one of which involves a $q$-analogue of the logarithmic functions $\log(\la)$ and $\log(1-\la)$, and compute a $q$-analogue of Wronskian of the explicit solutions.
	In \S\ref{DFandConn}, we construct a $q$-deformation of $\nabla$ on $H_{B'}$ which is related to the $q$-differential equation~\eqref{q-hyp-geom-eq} similarly to the relation between $\nabla$ on $H_{B'}$ and the differential equation~\eqref{hyp-geom-eq} recalled above.
	In \S\ref{Dwsec}, we show a $(p,q-1)$-adic formal congruence, which is a $q$-analogue of Dwork's results in~\cite[\S1--\S3]{Dw}. 
	In \S\ref{proofsec} and \S\ref{sec:answer}, by applying the formal congruence to the explicit solutions constructed in \S\ref{qsolu_sec}, we give Frobenius structures to the $q$-deformations of the connections on $U_{B'}$ and $M_{B''}\otimes_{B''}\Z_p[[\lambda]]$ (constructed in \S\ref{DFandConn}), and show that they give the desired canonical $q$-deformations. In \S\ref{FurtherTopic}, we further show that the $q$-deformation of $\varphi$ and $\nabla$ on $U_{B'}$ admits an ``arithmetic $\Gamma$-structure''.
	
	\begin{rem} 
		The $B''$-module $H_{B''}$ with $\nabla$, $\varphi_{B''}$, and the filtration is an object of $\text{MIC}_{[0,1]}(B'',\varphi,\Fil^{\bullet})$, and the functor~\eqref{intro:AR} in \S\ref{background} for $a=1$ and $R'=S''\coloneqq\Z_p[[q-1]]\left\langle\la,\frac{1}{(1-\la)h(\la)}\right\rangle$ gives a canonical $q$-deformation of $H_{B''}$ in $\text{MF}_{[0,1]}^{\q{p},q-1}(S'',\varphi,\Gamma)$. Therefore one may ask whether its $q$-connection is related to the $q$-differential equation~\eqref{q-hyp-geom-eq} similarly to the relation between $\nabla$ on $H^1_{\text{dR}}$ and the differential equation~\eqref{hyp-geom-eq}. We can also apply the same construction to the log smooth extension of the Legendre family over the base $\Z_p\left\langle \la,\frac{1}{1-\la}\right\rangle$ (without removing the supersingular locus), and ask the same question. It is natural to expect that this canonical $q$-deformation coincides with the conjectured (log) $q$-de Rham cohomology of the family. (We can compare the two in the category $\text{MF}_{[0,1]}^{\q{p},q-1} \left(\Z_p[q-1]/(q-1)^2\left\langle\lambda,\frac{1}{1-\la}\right\rangle,\varphi,\Gamma\right)$, where our canonical $q$-deformation is reduced to the scalar extension by $\Z_p\left\langle\lambda,\frac{1}{1-\la}\right\rangle\to \Z_p[q-1]/(q-1)^2\left\langle\lambda,\frac{1}{1-\la}\right\rangle$.) Thus our question is connected with the question by Scholze mentioned above.
	\end{rem}
	
	\textbf{Notation.}
	We fix some notation used throughout this paper.
	Let $p$ be an odd prime number.
	Let $v_p$ be the $p$-adic valuation of $\overline{\Q_p}$ normalized by $v_p(p)=1$.
	Let $q$ be a formal variable.
	Let $R=\Z_p[[q-1]]$.
	Let $Q$ be the quotient field of $R$.
	Let $B=\Z_p \left\langle \la, \frac{1}{\la (1-\la)h(\la)} \right\rangle$.
	Let $B'=\Z_p \left\langle \la, \frac{1}{h(\la)} \right\rangle$.
	We equip $B$ and $B'$ with the $p$-adic topology.
	Let $S'=R \left\langle \la, \frac{1}{h(\la)} \right\rangle$.
	We equip $\Rla$ and $S'$ with the $(p,q-1)$-adic topology.
	For $a \in \Q \cap \Z_p$, let $\q{a}$ be the $q$-number (the $q$-analogue of the rational number $a$) defined by
	\[
	\q{a} \coloneqq \frac{q^a -1}{q-1}=\sum_{i=1}^{\infty} \binom{a}{i} (q-1)^{i-1}.
	\]
	If $a$ is a positive integer, $\q{a}$ is equal to $1+q+q^2+\cdots +q^{a-1}$.
	
	\textbf{Acknowledgments.}
	This paper is based on the author's Master's thesis. The author is deeply grateful to Professor Takeshi Tsuji. He is the supervisor of the author.
	
	\section{\texorpdfstring{$q$}{q}-connection}
	\label{sec:basic}
	We define $\Rcompla$ to be the completion of the polynomial ring $R[\la]$ with respect to the $(p,q-1)$-adic topology, namely,
	\[
	R\left\langle \la \right\rangle \coloneqq \varprojlim_n R[\la] / {(p,q-1)}^n R[\la].
	\]
	For $g(\la) \in R[\la] \setminus (p,q-1) R[\la]$, we define $R\left\langle \la, \frac{1}{g(\la)} \right\rangle$ to be the completion of the ring $R \left[\la, \frac{1}{g(\la)}\right]$ with respect to the $(p,q-1)$-adic topology.
	In this section, we construct a $q$-analogue of the differential operator $\diff$ on $R'=\Rla,\Rcompla$ and $R\left\langle \la, \frac{1}{g(\la)} \right\rangle$ ($g(\la) \in R[\la]\setminus (p,q-1) R[\la]$).
	
	\begin{defi}
		We define the Frobenius endomorphism $\varphi$ of $\Rla,\Rcompla$ and $R\left\langle \la, \frac{1}{g(\la)} \right\rangle$ ($g(\la) \in R[\la]\setminus (p,q-1) R[\la]$) as follows:
		First, we define the endomorphism $\varphi$ of $R[\la]$ satisfying $\varphi(x) \equiv x^p \mod p$ by $\varphi(a)=a$ ($a\in \Z_p$), $\varphi(q)=q^p$, and $\varphi(\la)=\la^p$.
		Since it maps the ideals $(\la)$ and $(p,q-1)$ of $R[\la]$ into themselves,
		we can define $\varphi$ of $\Rla$ and $\Rcompla$ by taking its $\lambda$- and $(p,q-1)$-adic completions, respectively.
		The endomorphism $\varphi$ of $R[\la]$ induces a homomorphism $\varphi \colon R \left[\la,\frac{1}{g(\la)}\right] \to R\left[\la,\frac{1}{\varphi(g(\la))}\right]$,
		and its $(p,q-1)$-adic completion gives the Frobenius endomorphism $\varphi$ on $R \left\langle \la,\frac{1}{g(\la)}\right\rangle$ by Lemma~\ref{lem:phi} below.
	\end{defi}
	
	\begin{lemm}
		\label{lem:phi}
		For $g(\la), h(\la) \in R[\la]$ satisfying
		$g(\la) \equiv h(\la)^n \not \equiv 0 \mod(p,q-1)R[\la]$
		for some integer $n>0$, we have
		\[
		R\left\langle \la, \frac{1}{g(\la)} \right\rangle
		= R\left\langle \la, \frac{1}{h(\la)} \right\rangle.
		\]
	\end{lemm}
	
	\begin{proof}
		The congruence $g(\lambda)\equiv h(\lambda)^n \mod (p,q-1)$ implies that the image of $g(\la)$ in the quotient	$\left. R\left\langle \la, \frac{1}{h(\la)} \right\rangle \middle/ (p,q-1) R\left\langle \la, \frac{1}{h(\la)} \right\rangle \right.$ is a unit.
		Since $R\left\langle \la, \frac{1}{h(\la)} \right\rangle$ is $(p,q-1)$-adically complete, $g(\la)$ is a unit of $R\left\langle \la, \frac{1}{h(\la)} \right\rangle $. By the same argument, we see that $h(\lambda)$ is a unit of $R\left\langle \la, \frac{1}{g(\la)}\right\rangle$. Thus the
		natural homomorphisms $R[\lambda]\to R\left\langle \la, \frac{1}{h(\la)} \right\rangle, R\left\langle \la, \frac{1}{g(\la)}\right\rangle$ extend to
		homomorphisms
		\[
		f\colon R\left\langle \la, \frac{1}{g(\la)} \right\rangle\to R\left\langle \la, \frac{1}{h(\la)} \right\rangle, \quad
		g\colon R\left\langle \la, \frac{1}{h(\la)} \right\rangle \to R\left\langle \la, \frac{1}{g(\la)} \right\rangle,
		\]
		which satisfy $g\circ f=\text{id}$ and $f\circ g=\text{id}$.
	\end{proof}
	
	\begin{defi}
		Let $\Gamma$ be a group isomorphic to $\Z$, and let $\gamma$ be a generator of $\Gamma$.
		We define the action of $\Gamma$ on $\Rla,\Rcompla$ and $R\left\langle \la, \frac{1}{g(\la)} \right\rangle$ ($g(\la) \in R[\la]\setminus (p,q-1) R[\la]$) as follows.
		First, we define the action of $\Gamma$ on $R[\la]$ by $\gamma(\la)=q\la$ and $\gamma(a)=a$ ($a\in R$).
		Since this action of $\Gamma$ preserves the ideals $(\lambda)$ and $(p,q-1)$ of $R[\la]$, it induces an action of $\Gamma$ on $\Rla$ and $\Rcompla$.
		The action of $\gamma$ on $R[\lambda]$ extends to an isomorphism $\gamma \colon R \left[\la,\frac{1}{g(\la)}\right] \to R\left[\la,\frac{1}{\gamma(g(\la))}\right]$,
		whose $(p,q-1)$-adic completion gives an action of $\gamma$ on $R\left\langle \la, \frac{1}{g(\la)} \right\rangle$ by Lemma~\ref{lem:phi}.
	\end{defi}
	
	Let $R'$ be one of the $R$-algebras $\Rla$, $\Rcompla$ and $R\left\langle \la, \frac{1}{g(\la)} \right\rangle$.
	Since $R'$ is noetherian, $(q-1)\la R'$ is a closed ideal of $R'$ with respect to the $\lambda$-adic topology for $R'=\Rla$, and the $(p,q-1)$-adic topology for $R'=\Rcompla, R\left\langle \la, \frac{1}{g(\la)} \right\rangle$.
	Hence $(\gamma-1)(\lambda^n)=(q-1)[n]_q\lambda^n\in (q-1)\lambda R[\lambda]$ and $(\gamma-1)(g(\lambda)^{-n})=-g(\lambda)^{-n}\gamma(g(\lambda))^{-n} (\gamma-1)(g(\lambda)^n)\in (q-1)\lambda R\left\langle \la, \frac{1}{g(\la)} \right\rangle$ imply the inclusion $(\gamma-1)(R')\subset (q-1)\lambda R'$.
	
	\begin{defi}
		\label{def:q-diff_operator}
		We define the $q$-differential operator $d_q \colon R' \to R'$ by $d_q=\frac{\gamma-1}{(q-1)\la}$.
	\end{defi}
	This is a $q$-analogue of the differential operator $\diff$.
	Clearly we have $\gamma=1+(q-1)\la d_q$.
	
	\begin{prop}
		\label{q-Leibniz}
		We have a $q$-analogue of Leibniz rule:
		\[
		d_q(xy)=d_q(x)\gamma(y)+xd_q(y), \quad x,y \in R'.
		\]
	\end{prop}
	
	\begin{proof}
		The formula follows from $(\gamma-1)(xy)= (\gamma-1)(x) \gamma(y) +x(\gamma-1)(y)$.
	\end{proof}
	
	\begin{defi}
		We define the $q$-differential module $q\Omega_{R'/R}$ to be the free $R'$-module $R' d\log\la$ and the $q$-derivation $\delta_q\colon R' \to q\Omega_{R'/R}$ by $\delta_q(x)=d_q(x)\cdot d\la$, where $d\la=\la d\log\la$.
	\end{defi}
	Let $A=R'/(q-1)R'$. Then $q\Omega_{R'/R} \bmod q-1$ is naturally identified with the differential module $\Omega_{A/\Z_p,\log}$ with log poles along the divisor $\la=0$.
	Since $\delta_q$ is $R$-linear, we can define $\delta \colon A \to \Omega_{A/\Z_p,\log}$ as $\delta_q \bmod q-1$. We have $\delta(x)=\diff x \cdot d\la$,
	i.e., $\delta$ is the universal continuous $R$-linear derivation.
	
	\begin{defi}
		Let $M$ be an $R'$-module.
		An $R$-linear map $\nabla_q \colon M \to M \otimes_{R'} q\Omega_{R'/R}$ is called a $q$-connection on $M$ if it satisfies
		\[
		\nabla_q(am)=m \otimes \delta_q(a) +\gamma(a)\nabla_q(m), \quad a \in R', m \in M.
		\]
	\end{defi}
	Let $(M,\nabla_q)$ be an $R'$-module with a $q$-connection.
	Let $A=R'/(q-1)R'$, $N=M/(q-1)M$ and
	let $\nabla \colon N \to N \otimes_A \Omega_{A/\Z_p,\log} $ be $\nabla_q \bmod q=1$. Then $\nabla$ satisfies
	\[
	\nabla(am)=m \otimes \delta(a) +a\nabla(m)
	\]
	for all $a \in A, m \in N$, i.e., $\nabla$ is a connection on $N$.
	
	\begin{defi}
		\label{def:q-connection}
		For	an $R'$-module $M$ with a $q$-connection $\nabla_q$, we define the $R$-linear endomorphisms $D^{\log}_q$ and $\gamma_M$ of $M$ by $\nabla_q =D^{\log}_q\otimes d\log\la$ and $\gamma_M = 1+(q-1)D_q^{\log}$. We define $D_q$ to be $\la^{-1}D^{\log}_q$ if $\la$ is invertible in $R'$.
	\end{defi}
	We see that $\gamma_M$ is $\gamma$-semilinear as follows:
	For all $r \in R' , m \in M $, we have
	\begin{align*}
	\gamma_M(rm) & =rm+(q-1) D^{\log}_q(rm) \\
	& =rm+(q-1)(\la d_q(r)m+\gamma(r) D_q^{\log}(m)) \\
	& =rm+(\gamma-1)(r)m+(q-1) \gamma(r) D_q^{\log}(m) \\
	& =\gamma(r)\gamma_M(m).
	\end{align*}
	The endomorphism $\gamma_M$ is bijective and defines a $\rho$-semilinear action of $\Gamma$ on $M$ continuous with respect to the $(p,q-1)$-adic topology.
	
	\begin{rem}
		\label{rem:q-connection} 
		For any $\ell(\lambda)\in R'\backslash\{0\}$, we can define a $q$-connection $\nabla_q\colon M\to M\otimes_{R'}\frac{1}{\ell(\lambda)}q\Omega_{R'/R}$ in the same way as in Definition~\ref{def:q-connection}. However we cannot construct the $\Gamma$-action associated to $\nabla_q$ in general unless $\ell(\lambda)$ is invertible in $R'$.
	\end{rem}
	
	We give some properties necessary for the $q$-analogue calculation in the following sections.
	
	\begin{prop}
		\label{prop:qanalogue}
		(i) $\q{a+b}=\q{a}+q^a \q{b}$; \quad
		(ii) $\q{ap}=\q{p} \varphi \left(\q{a} \right)$; \label{lem:qap} \quad
		(iii) $\q{p^n} \in (p,q-1)^n R$ \label{lem:qp^n};\\
		(iv) $\gamma \circ \varphi= \varphi \circ \gamma$ on $R'$; \label{lem:circ} \quad
		(v) $d_q \varphi=\q{p}\la^{p-1} \varphi d_q $ on $R'$. \label{lem:dqvarphi}
	\end{prop}
	
	\begin{proof}
		We can verify the equalities (i) and (ii) by simple computations.
		The claim (iii) follows from $[p^n]_q\equiv p^n \mod (q-1)R$.
		For the claim (iv), since both sides are continuous endomorphisms, it suffices to show the commutativity for $a\in \Z_p$, $q$, and $\la$, which is verified as follows:
		\[
		\gamma \circ \varphi(a)=a= \varphi \circ \gamma(a),\quad
		\gamma \circ \varphi(q)=q^p= \varphi \circ \gamma(q),\quad
		\gamma \circ \varphi(\la)=q^p \la^p= \varphi \circ \gamma(q).
		\]
		We obtain the equality (v) by substituting $\gamma=1+(q-1)\lambda d_q$ into the equality (iv).
	\end{proof}
	
	\section{Background}
	\label{background}
	First, we define a category $\text{MIC}_{[0,a]}(A,\varphi,\Fil^{\bullet})$. Let $a$ be a non-negative integer. Let $A$ be a $p$-torsion free commutative ring with an endomorphism $\varphi$, and let $\Fil^{\bullet} A$ be the trivial decreasing filtration of $A$ defined by
	\begin{equation*}
	\Fil^r A= \begin{cases}
	A & r \in \Z \cap \left( -\infty,0 \right] \\
	0 & r \in \Z \cap \left( 0, \infty \right).
	\end{cases}
	\end{equation*}
	Let $\Omega_{A}$ be a free $A$-module of rank $1$,
	let $\delta \colon A \to \Omega_{A}$ be a derivation, and let $\varphi^1\colon \Omega_A\to\Omega_A$ be a $\varphi$-semilinear homomorphism satisfying $\varphi^1\circ\delta=\delta\circ\varphi$.
	
	\begin{defi}
		We define the category $\text{MIC}_{[0,a]}(A,\varphi,\Fil^{\bullet})$ as follows.
		An object is a quartet $(M, \Fil^{\bullet} M, \varphi_M, \nabla)$ consisting of the following.
		\begin{enumerate}[(i)]
			\item $M$ is a free $A$-module of finite type. (Let $N$ be the rank of $M$.)
			\item A decreasing filtration $\Fil^r M$ ($r \in \N \cap [0,a]$) of $M$ satisfying the following conditions.
			\begin{enumerate}
				\item[({ii}-a)] There exists a basis $e_\nu$ ($\nu \in \N \cap [1,N]$) of $M$ and $r_\nu \in \N \cap [0,a]$ for each $\nu \in \N \cap [1,N]$ such that $\Fil^r M = \oplus_{\nu \in \N \cap [1,N]} \Fil^{r-r_{\nu}} A e_{\nu}$ for $r \in \N \cap [0,a]$.
			\end{enumerate}
			\item $\varphi_M \colon M \to M$ is a $\varphi$-semilinear endomorphism satisfying the following condition.
			\begin{enumerate}[({iii}-a)]
				\item $\varphi_M (\Fil^r M) \subset p^{r} M$ for $r \in \N \cap [0,a]$.
				\item $M=\sum_{r \in \N \cap [0,a]} A\cdot p^{- r} \varphi_M (\Fil^r M)$.
			\end{enumerate}
			\item $\nabla \colon M \to M \otimes_A \Omega_{A}$ is a connection on $M$ satisfying the following condition.
			\begin{enumerate}[({iv}-a)]
				\item $\nabla(\Fil^r M) \subset \Fil^{r-1} M \otimes_{ A} \Omega_{A}$ for $r \in \N \cap [1,a]$.
				\item $\nabla\circ\varphi_M=(\varphi_M\otimes\varphi^1)\circ \nabla$.
			\end{enumerate}
		\end{enumerate}
		A morphism is an $A$-linear homomorphism preserving the filtration, and compatible with $\varphi_M$ and $\nabla$.
	\end{defi}
	
	The Frobenius structure $\varphi_{ H^1_{\text{dR}}}$ recalled in \S\ref{introduction} satisfies the following.
	
	\begin{enumerate}[(i)]
		\item Assume that $a=0$, $A=B'$, and $\Omega_A=Ad\la$.
		Then we have $\left(U_{B'}, \nabla, \varphi_{U_{B'}}\right)\in \text{MIC}_{[0,0]}(A,\varphi,\Fil^{\bullet})$.
		(In the case $a=0$, we can forget filtrations because the condition~(ii-a) implies $\Fil^0 U=U$ and $\Fil^1 U=0$.)
		\item Assume that $a=1$, $A=\Z_p[[\la]]$, and $\Omega_A=Ad\log\la$.
		Put $H_0 =H_{B''} \otimes_{B''}\Z_p[[\la]]$, which has the connection $\nabla$, the Frobenius structure $\varphi_{H_0}$, and the filtration induced by those of $H_{B'}$.
		Then we have $\MICHB \in \text{MIC}_{[0,1]}(A,\varphi,\Fil^{\bullet}).$
		\item Assume that $a=1$, $A=B''$, and $\Omega_A=Ad\log\la$. Then we have $\left(H_{B''},\nabla,\varphi_{H_{B''}},\Fil^{\bullet}\right)\in \text{MIC}_{[0,1]}(A,\varphi,\Fil^{\bullet}).$
	\end{enumerate}
	
	Next, we define a category $\text{MF}_{[0,a]}^{\q{p},\text{cont}}(R',\varphi,\Gamma)$.
	Let $a$ be the same as above.
	As before Definition~\ref{def:q-diff_operator}, let $R'$ be one of the rings $R[[\la]]$, $R\langle\la\rangle$, and $R\left\langle\la, \frac{1}{g(\la)}\right\rangle$ $(g(\la)\in R[\la]\backslash(p,q-1)R[\la])$ equipped with the $(p,q-1)$-adic topology, and put $A=R'/(q-1)$.
	Then $R'$ is a commutative ring with an endomorphism $\varphi$ and an action $\rho$ of $\Gamma$ (See \S\ref{sec:basic}).
	Let $\Fil^{\bullet} R'$ be the decreasing filtration of $R'$ defined by
	\begin{equation*}
	\Fil^r R'= \begin{cases}
	R' & r \in \Z \cap (-\infty,0 ) \\
	(q-1)^r R' & r \in \Z \cap \left[ 0,\infty \right).
	\end{cases}
	\end{equation*}
	We give $\Gamma$ a discrete topology. Then, as $\varphi(q-1)=(q-1) \q{p}$ and $\gamma \circ \varphi =\varphi \circ \gamma $, the quartet $(R', \q{p}, \Fil^{\bullet} R', \varphi, \rho )$ satisfies all conditions in $\S 6,7$ of~\cite{Tsu}.
	
	Thus we can define the category $\text{MF}_{[0,a]}^{\q{p},\text{cont}}(R',\varphi,\Gamma)$.
	
	\begin{defi}[{\cite[\S7]{Tsu}}]
		\label{def:MF}
		We define the category $\text{MF}_{[0,a]}^{\q{p},\text{cont}}(R',\varphi,\Gamma)$ as follows.
		An object is a quartet $(M, \Fil^{\bullet} M, \varphi_M, \rho_M )$ consisting of the following.
		\begin{enumerate}[(i)]
			\item $M$ is a free $R'$-module of finite type. (Let $N$ be the rank of $M$.)
			\item A decreasing filtration $\Fil^r M$ ($r \in \N \cap [0,a]$) of $M$ satisfying the following
			conditions.
			\begin{enumerate}
				\item[({ii}-a)] There exists a basis $e_\nu$ ($\nu \in \N \cap [1,N]$) of $M$ and $r_\nu \in \N \cap [0,a]$ for each $\nu \in \N \cap [1,N]$ such that
				$\Fil^r M = \oplus_{\nu \in \N \cap [1,N]} \Fil^{r-r_{\nu}} R' e_{\nu}$ for $r \in \N \cap [0,a]$.
			\end{enumerate}
			\item $\varphi_M \colon M \to M$ is a $\varphi$-semilinear endomorphism satisfying the following conditions.
			\begin{enumerate}[({iii}-a)]
				\item $\varphi_M (\Fil^r M) \subset \q{p}^r M$ for $r \in \N \cap [0,a]$.
				\item $M=\sum_{r \in \N \cap [0,a]} R'\cdot \q{p}^{-r} \varphi_M (\Fil^r M)$.
			\end{enumerate}
			\item $\rho_M \colon \Gamma \to \text{Aut}(M)$ is a $\rho$-semilinear action and satisfies the following conditions.
			\begin{enumerate}[({iv}-a)]
				\item $\rho_M(g)(\Fil^r M)=\Fil^r M $ for $r \in \N \cap [0,a]$ and $g \in \Gamma$.
				\item $\rho_M(g) \varphi_M =\varphi_M \rho_M(g)$ for $g \in \Gamma$.
			\end{enumerate}
			\item $\Gamma \times M \to M ; (g,m) \mapsto \rho_M(g)m $ is continuous.
			\label{def:MFconti}
		\end{enumerate}
		A morphism is an $R'$-linear homomorphism preserving the filtrations, compatible with $\varphi_M$'s , and moreover $\Gamma$-equivariant.
	\end{defi}
	
	\begin{rem}
		\label{rem:q-connectionFrob}
		Let $M$ be a free $R'$-module equipped with a $q$-connection $\nabla_q\colon M\to M\otimes_{R'}q\Omega_{R'/R}$, and a $\varphi$-semilinear endomorphism $\varphi_M$. Let $\rho_M$ be the $\rho$-semilinear continuous action of $\Gamma$ on $M$ associated to $\nabla_q$. Then $\varphi_M$ is $\Gamma$-equivariant, i.e., satisfy the condition (iv-b) in Definition~\ref{def:MF} if and only if $(\varphi_M\otimes\varphi^1)\circ\nabla_q =\nabla_q\circ\varphi_M$. Here $\varphi^1$ is the $\varphi$-semilinear endomorphism of $q\Omega_{R'/R}$ defined by $\varphi^1(d\log\lambda)=\q{p}d\log\lambda$. Note that we have $\varphi^1\circ\delta_q=\delta_q\circ\varphi$ by Proposition~\ref{prop:qanalogue} (v).
	\end{rem}
	
	\begin{defi}
		We define $\text{MF}_{[0,a]}^{\q{p},q-1}(R',\varphi,\Gamma)$ to be the full subcategory of $\text{MF}_{[0,a]}^{\q{p},\text{cont}}(R',\varphi,\Gamma)$ consisting of $M$ such that the $\Gamma$-action $\rho_M$ on $M$ satisfies$ (\rho_M(\gamma)-1)(M) \subset (q-1) M. $
	\end{defi}
	The Frobenius $\varphi$ and the $\Gamma$-action on $R'$ induce those on $R'/(q-1)^{a+1}$, and we can define the decreasing filtration of $R'$ by $\Fil^{\bullet} (R'/(q-1)^{a+1}) = (\Fil^{\bullet} R') \cdot (R'/(q-1)^{a+1}) $. 
	Therefore we can define the categories $\text{MF}_{[0,a]}^{\q{p},\text{cont}}(R'/(q-1)^{a+1},\varphi,\Gamma)$ and $\text{MF}_{[0,a]}^{\q{p},q-1}(R'/(q-1)^{a+1},\varphi,\Gamma)$ in the same way.
	
	We equip $A$ with the derivation $\delta\colon A\to \Omega_{A,\log}=Ad\log\la$, the reduction mod $q-1$ of $\delta_q\colon R'\to q\Omega_{R'/R}$. Then the scalar extension by $R'\to A$ induces a functor
	\begin{equation}
	\label{mod_q-1_functor}
	\text{MF}_{[0,a]}^{\q{p}, q-1}(R',\varphi,\Gamma)
	\xrightarrow{\bmod q-1 }
	\text{MIC}_{[0,a]}(A,\varphi,\Fil^{\bullet}).
	\end{equation}
	(For $(M, \Fil^{\bullet} M, \varphi_M, \rho_M ) \in \text{MF}_{[0,a]}^{\q{p}, q-1}(R',\varphi,\Gamma)$,
	we define the connection of $M/(q-1)M$ by $(\frac{\rho_M(\gamma)-1}{q-1} \bmod q-1)\otimes d\log\la$.)
	
	For $a=0,1$, we have a functor induced by the base change from $A$ to $R'/(q-1)^{a+1}$:
	\begin{equation}
	\label{intro:ARJ}
	\text{MIC}_{[0,a]}(A,\varphi,\Fil^{\bullet})
	\to
	\text{MF}_{[0,a]}^{\q{p}, q-1}(R'/(q-1)^{a+1},\varphi,\Gamma).
	\end{equation}
	For $(M, \Fil^{\bullet} M, \varphi_M, \nabla) \in \text{MIC}_{[0,a]}(A,\varphi,\Fil^{\bullet})$,
	we define the $\rho$-{semilinear action of $\Gamma$ on $M\otimes_A R'/(q-1)^{a+1}$ by $1+(q-1)D^{\log}_q$ if $a=1$, and by $1$, i.e., the trivial action if $a=0$. 
	For the Frobenius structure, note that we have $\q{p}=p\cdot(\text{unit})$ in $R/(q-1)^{a+1}$ because $a+1\leq p-1$.
	The reduction mod $(q-1)^{a+1}$ gives an equivalence of categories~\cite{Tsu}
	\begin{equation}
	\label{intro:RJ}
	\text{MF}_{[0,a]}^{\q{p},\text{cont}}(R',\varphi,\Gamma)
	\to
	\text{MF}_{[0,a]}^{\q{p},\text{cont}}(R'/(q-1)^{a+1},\varphi,\Gamma),
	\end{equation}
	which induces an equivalence between full subcategories
	\begin{equation}
	\label{intro:RJ2}
	\text{MF}_{[0,a]}^{\q{p},q-1}(R',\varphi,\Gamma)
	\to
	\text{MF}_{[0,a]}^{\q{p},q-1}(R'/(q-1)^{a+1},\varphi,\Gamma),
	\end{equation}
	By combining~\eqref{intro:ARJ} and~\eqref{intro:RJ2}, we obtain a right inverse of the functor~\eqref{mod_q-1_functor}
	\begin{equation}
	\label{intro:AR}
	\text{MIC}_{[0,a]}(A,\varphi,\Fil^{\bullet})
	\to
	\text{MF}_{[0,a]}^{\q{p}, q-1}(R',\varphi,\Gamma).
	\end{equation}
	If $a=0$, then the functor~\eqref{intro:ARJ} is an equivalence of categories.
	Hence the functors~\eqref{mod_q-1_functor} and~\eqref{intro:AR} are equivalences of categories quasi-inverse of each other.
	By applying this functor, we obtain the canonical $q$-deformations of the objects $(U_{B'},\nabla,\varphi_{U_{B'}})$ and $\MICHB$ to $\text{MF}_{[0,a]}^{\q{p},q-1}(R',\varphi,\Gamma)$ for $R'=S'$, $a=0$ and $R'=R[[\la]]$, $a=1$, respectively.
	
	We want to know whether the two canonical deformations are related to the $q$-hypergeometric differential equation~\eqref{q-hyp-geom-eq}.
	
	\begin{rem}
		To show the equivalence of categories~\eqref{intro:RJ}, we have to check that $R'$ with the $\Gamma$-action, $\varphi$, and the filtration satisfies the conditions in~\cite[{\S6,\S7}]{Tsu}, specifically~\cite[Conditions 39 and 54]{Tsu}, while they are trivial.
	\end{rem}
	
	\section{Main theorems}
	\label{Main}
	As before Definition~\ref{def:q-diff_operator}, let $R'$ be one of the rings $R[[\la]]$, $R\langle\la\rangle$, and $R\left\langle\la, \frac{1}{g(\la)}\right\rangle$ $(g(\la)\in R[\la]\backslash(p,q-1)R[\la])$ equipped with the $(p,q-1)$-adic topology.
	Let $\nabla_q \colon M' \to M' \otimes_{R'} \frac{1}{1-\la}q\Omega_{R'/R}$ be the $q$-connection associated with the $q$-hypergeometric differential equation on a free $R'$-module $M'$ of rank $2$ introduced later in \S\ref{DFandConn}.
	The $R'$-module $M'$ is also equipped with a filtration, and the reduction mod $q-1$ of $(M',\nabla_q,\Fil^\bullet)$ is canonically isomorphic to $(H_{B'}\otimes_{B'}A,\nabla,\Fil^\bullet)$ (see~\eqref{qHGconnectionModq-1}):
	\begin{equation}
	\label{Mprime_mod_q-1}
	(M',\nabla_q,\Fil^\bullet)\otimes_{R'}A\cong(H_{B'}\otimes_{B'}A,\nabla,\Fil^\bullet).
	\end{equation}
	Let $S'=R\left\langle \la,\frac{1}{h(\la)}\right\rangle $ be the $(p,q-1)$-adic completion of $R\left[ \la,\frac{1}{h(\la)}\right] $, where $h(\la)$ is the Hasse polynomial (see \S\ref{introduction}).
	In this paper, we give an appropriate Frobenius structure which is compatible with the $q$-connection on a rank $1$ and $\nabla_q$-stable submodule of $M'$ (resp. $M'$ itself) in the case $R'=S'$ (resp. $\Rla$).
	\begin{theo}
		\label{Main1}
		Assume that $R'=S'$.
		There exists a pair $(U',\varphi_{U'})$ which satisfies the following conditions.
		\begin{enumerate}[(i)]
			\item $U'$ is a direct factor of the $S'$-module $M'$ free of rank $1$ satisfying $\nabla_q(U')\subset U'\otimes_{R'}q\Omega_{R'/R}$. Let $\rho_{U'}$ be the $\rho$-semilinear continuous action of $\Gamma$ on $U'$ associated to $\nabla_q\vert_{U'}$.
			\item $\varphi_{U'}$ is a $\varphi$-semilinear endomorphism of $U'$
			and satisfies $S'\cdot\varphi_{U'}(U')=U'$. \label{const:FphiF}
			\item The triple $(U',\varphi_{U'},\rho_{U'})$ is an object of $\text{MF}_{[0,0]}^{\q{p}, q-1}(S', \varphi, \Gamma)$, i.e., $\rho_{U'}(\gamma)\circ \varphi_{U'}=\varphi_{U'} \circ \rho_{U'}(\gamma)$.
			\item The isomorphism~\eqref{Mprime_mod_q-1} induces a $B'$-linear isomorphism $U'\otimes_{S'}B'\xrightarrow{\cong}U_{B'}$, and it defines an isomorphism between $(U_{B'},\nabla,\varphi_{U_{B'}})$ and the image of $(U',\varphi_{U'},\rho_{U'})$ under the equivalence of categories (see after~\eqref{intro:AR})
			\[
			\text{MF}_{[0,0]}^{\q{p}, q-1}({S'}, \varphi, \Gamma) \to
			\text{MIC}_{[0,0]}(B',\varphi,\Fil^{\bullet}).
			\]
		\end{enumerate}
	\end{theo}
	
	\begin{theo}
		\label{Main2}
		Assume that $R'=\Rla$, and let $\rho_{M'}$ be the continuous $\rho$-semilinear action of $\Gamma$ on $M'$ associated to $\nabla_q$. (Note that we have $\frac{1}{1-\la}q\Omega_{R'/R}=q\Omega_{R'/R}$ because $1-\la\in R[[\lambda]]^{\times}$.)
		Then there exists a $\varphi$-semilinear endomorphism $\varphi_{M'}$ on $M'$ which satisfies the following conditions.
		\begin{enumerate}[(i)]
			\item There exists a basis $(f_1, f_2)$ of $M'$ such that $\varphi( e_1')=\q{p}f_1, \varphi(e_2)=f_2$. \label{Main2:cond1}
			\item The quartet $(M',\Fil^{\bullet}M',\varphi_{M'},\rho_{M'})$ is an object of $\text{MF}_{[0,1]}^{\q{p},q-1}(\Rla,\varphi,\Gamma)$, i.e., $\rho_{M'}(\gamma)(\Fil^1M')=\Fil^1M'$ and $\rho_{M'}(\gamma)\circ \varphi_{M'}=\varphi_{M'} \circ \rho_{M'}(\gamma)$.\label{Main2:cond2}
			\item The isomorphism~\eqref{Mprime_mod_q-1} induces an isomorphism
			\[
			(M',\varphi_{M'},\nabla_q,\Fil^{\bullet}M')\otimes_{R[[\la]]}\Z_p[[\la]]
			\cong \MICHB.
			\]
			Moreover, it can be lifted to an isomorphism between $(M',\Fil^{\bullet}M',\varphi_{M'},\rho_{M'}$) and the canonical $q$-deformation of $\MICHB$ constructed after~\eqref{intro:AR}.
		\end{enumerate}
	\end{theo}
	
	\section{\texorpdfstring{$q$}{q}-hypergeometric differential equation and its solutions}
	\label{qsolu_sec}
	In this section, we give explicit solutions of the $q$-hypergeometric differential equation~\cite{BHS} defined by
	\begin{equation}
	\label{eq:qhyp}
	L[f]= q\lambda(1-q\lambda) d^2_q{f} +(1-(1+\q{2}-2\q{\tfrac{1}{2}}) \la) d_q{f} -\q{\tfrac{1}{2}}^2 f=0,
	\end{equation}
	which is a $q$-analogue of the classical hypergeometric differential equation~\cite{Put}
	\[
	\la(1-\la) \dtwo f +(1-2\la) \diff f -\frac{1}{4} f=0.
	\]
	For convenience, we put $\alpha = 1+\q{2}-2\q{\tfrac{1}{2}}$.
	
	To describe the solutions, we introduce $q$-logarithmic function $\log_q(\text{-})$. (Note that the ``$q$" in ``$\log_q$" does not mean a base.)
	Since $d_q$ does not have compatibility with the translation $\la \mapsto \la + a$ for $a \in \Z_p$, we have to define $\log_q \la $ and $\log_q(1-\la)$ respectively.
	Since $\varphi$ and $\gamma$ are injective as endomorphisms of $\Rla$, we can extend them to endomorphisms of $Q((\la))$.
	First, put $ \log_q(1-\la)=- \sum_{n=1}^{\infty} \frac{\la^n}{\q{n}}.$
	Then, $\log_q(1-\la)$ is an element of $Q((\la))$ and $d_q \log_q (1-\la) =\frac{1}{\la-1}$.
	Next, let $\log_q \la$ be a formal variable and extend $\varphi$ and $\gamma$ to endomorphisms of $\Qlalog$ by
	\[
	\varphi(\log_q \la) =\q{p} \log_q \la \quad\text{and} \quad \gamma(\log_q \la) =q-1+ \log_q \la.
	\]
	Then, by
	\[
	\gamma(\varphi(\log_q \la))=(q-1)\q{p}+\q{p} \log_q \la=\varphi(\gamma(\log_q \la)),
	\]
	the commutativity $\gamma \circ \varphi =\varphi \circ \gamma $ is satisfied on $\Qlalog$. Moreover, since $(q-1)\la\in Q((\la))^{\times} $, we can extend $d_q$ of $Q((\la))$ to $\Qlalog$ by $d_q=\frac{\gamma-1}{(q-1)\la}$. We have $d_q \log_q \la =\frac{1}{\la}.$
	In the theorem below, we give explicit solutions of~\eqref{eq:qhyp} in $\Qlalog$.
	Put $a_n=\prod^{n-1}_{i=0} \left(\frac{[i+\frac{1}{2}]_q}{[i+1]_q} \right)^2$ for a non-negative integer $n$.
	
	\begin{theo}
		\label{qhyp_solution}
		We have $L[F]=L[H]=0$, where
		\[
		F=\sum^{\infty}_{n=0} a_n \lambda^n, \qquad
		H = F \log_q \la-F \log_q (1- \la)-
		\sum^{\infty}_{n=1}a_n \lambda^n \sum^{n}_{i=1} \left(\frac{2}{[i]_q} +q-1\right).
		\]
	\end{theo}
	
	First, we give some lemma.
	\begin{lemm}
		\label{lem:gd}
		(i) $d_q \gamma =q \gamma d_q$ \label{lem:gd1}; \quad
		(ii)$d_q \gamma + \gamma d_q = (1+q)\left(d_q+(q-1)\la d_q^2\right) $ \label{lem:gd2}; \\
		(iii)$\gamma^2 = 1+ (q^2-1)\la d_q +q(q-1)^2 \la^2 d_q^2 \label{lem:gd3}$.
	\end{lemm}
	
	\begin{proof}
		We can verify the equalities by simple computations.
	\end{proof}
	
	\begin{proof}[Proof of Theorem~\ref{qhyp_solution}]
		First, we prove that $L[F]=0$.
		Put $y=\sum_{n=0}^{\infty} c_n \la^n \in Q[[\la]]$ ($c_n \in Q$). Then, the coefficient of $\la^n$ in $-\q{\tfrac{1}{2}}^2 y$ , in $(1-\alpha \la)d_q y$, and in $q\la (1-q\la) d_q^2 y$ are $-\q{\tfrac{1}{2}}^2 c_n$, $\q{n+1} c_{n+1} - \alpha \q{n} c_n $, and $q \q{n+1} \q{n} c_{n+1} -q^2 \q{n} \q{n-1} c_n $ respectively.
		By adding all of them, we see that the coefficient of $\la^n$ in $L[y]$ is $\q{n+1}^2 c_{n+1} - \q{n+\tfrac{1}{2}}^2 c_n$.
		Since $\left\lbrace a_n \right\rbrace _{n\geq 0}$ has the property $\q{n+1}^2 a_{n+1} = \q{n+\tfrac{1}{2}}^2 a_n$ for $n \geq 0$, we obtain $L[F]=0$.
		
		For a non-negative integer $r$, we put $F_{\geq r}=\sum_{n=r}^{\infty} a_n \la^n.$
		Then by the same calculation as $L[F]$, we have $L[F_{\geq r}]=\q{r}^2 a_r \la^{r-1}.$
		To prove $L[H]=0$, we calculate $L[F\log_q \la - F\log_q (1- \la)]$ by using the following two claims.
		
		\begin{claim*}
			\begin{equation}
			L[F\log_q \la]=-2 \q{\tfrac{1}{2}}F + 2(1-q^\frac{1}{2}\la ) d_q F .
			\label{claimLF1}
			\end{equation}
		\end{claim*}
		First, we calculate $d_q (F \log_q\la )$ and $d_q^2 (F \log_q\la )$.
		\[
		d_q (F \log_q\la )
		= d_q F \log_q\la+\frac{1}{\la} F+(q-1) d_q F,
		\]
		\begin{align*}
		d^2_q (F \log_q\la ) & = d^2_q F \log_q\la+ (\gamma d_q +d_q\gamma ) F\frac{1}{\la} -\gamma^2(F)\frac{1}{q\la^2} \\
		& = d^2_q F \log_q\la -\frac{1}{q\la^2}F +\frac{q+1}{q\la} d_q F +2(q-1)d^2_q F.
		\end{align*}
		Thus,
		\begin{align*}
		L[F\log_q\la]
		& =q\la(1-q\la)d_q^2(F \log_q\la )+(1-\alpha \la)d_q (F \log_q\la )-\q{\tfrac{1}{2}}^2 F\log_q\la \\
		& = q\la(1-q\la) \left(-\frac{1}{q\la^2}F +\frac{q+1}{q\la} d_q F +2(q-1)d^2_q F\right) \\
		& \quad +(1-\alpha \la) \left( \frac{1}{\la} F+(q-1) d_q F \right) \qquad (\text{by }L[F]=0) \\
		& =\left( q-\alpha +2(q-1)\q{\tfrac{1}{2}}^2 \right) F+\left( 2+(-q^2-q+\alpha (q-1))\la \right) d_qF \\
		& =-2 \q{\tfrac{1}{2}}F + 2(1-q^\frac{1}{2}\la ) d_q F .
		\end{align*}
		
		\begin{claim*}
			\begin{equation}
			L[F\log_q (1-\la)]=-\frac{2q \q{\tfrac{1}{2}} \la -1}{q\la-1}F-2q\la \frac{q^\frac{1}{2}\la-1}{q \la -1}d_qF.
			\label{claimLF2}
			\end{equation}
		\end{claim*}
		In the same way as above, we obtain
		\[
		d_q (F \log_q(1-\la) )= d_q F \log_q(1-\la)+\frac{1}{\la-1} F+\frac{(q-1)\la}{\la-1} d_q F,
		\]
		\[
		d^2_q (F \log_q(1-\la) )=d^2_q F \log_q(1-\la) -\frac{1}{(\la-1)(q\la-1) }F +\frac{q+1}{q\la-1}d_qF +\frac{(q-1)\la(2q\la-q-1)}{(\la-1)(q\la-1)}d_q^2F.
		\]
		Thus,
		\begin{align*}
		& \quad L[F\log_q(1-\la)] \\
		& = q\la(1-q\la) \Bigg( -\frac{1}{(\la-1)(q\la-1) }F +\frac{q+1}{q\la-1}d_qF +\frac{(q-1)\la(2q\la-q-1)}{(\la-1)(q\la-1)}d_q^2F \Bigg) \\
		& \quad +(1-\alpha \la) \left( \frac{1}{\la-1} F+\frac{(q-1)\la}{\la-1} d_q F \right) \qquad (\text{by }L[F]=0) \\
		& =\frac{(1+(q-\alpha)\la)(q\la-1)+\q{\frac{1}{2}}^2 (q-1)\la (2q\la-q-1)}{(\la-1)(q\la-1)}F \\
		& \quad + \frac{((1-\alpha \la)(q-1)\la-q(q+1)\la(\la-1))(q\la-1)-(1-\alpha \la)(q-1)\la(2q\la-q-1) }{(\la-1)(q\la-1)}d_qF. \\
		& = -\frac{2q \q{\tfrac{1}{2}} \la -1}{q\la-1}F-2q\la \frac{q^\frac{1}{2}\la-1}{q \la -1}d_qF.
		\end{align*}
		Now the proofs of the two claims are completed. By claims~\eqref{claimLF1} and~\eqref{claimLF2},
		\begin{align}
		& \quad L[F\log_q\la - F\log_q(1-\la)] \notag \\
		& = 2(1-q^\frac{1}{2}\la ) d_q F -2 \q{\tfrac{1}{2}}F -
		\left( -2q\la \frac{q^\frac{1}{2}\la-1}{q \la -1}d_qF-\frac{2q \q{\tfrac{1}{2}} \la -1}{q\la-1}F \right) \notag \\
		& = 2\frac{q^\frac{1}{2}\la-1}{q \la -1}d_qF+\frac{2\q{\tfrac{1}{2}}-1}{q\la-1}F.
		\label{eq:LFlog}
		\end{align}
		We describe~\eqref{eq:LFlog} as an $Q$-linear combination of $L[F_{\geq n+1}]$ for $n$ (and $d_qF$).
		We start by writing $\frac{1}{1-q\la}F$ and $\frac{q\la}{1-q\la}d_qF$ as $Q$-linear combinations of $L[F_{\geq n+1}]$.
		\[
		\frac{1}{1-q\la}F =\left( \sum_{n=0}^{\infty} a_n \lambda^n \right)
		\left( \sum_{m=0}^{\infty} q^m \lambda^m \right)
		= \sum_{n=0}^{\infty} \la^n \left( \sum_{m=0}^n a_m q^{n-m} \right)
		= \sum_{n=0}^{\infty} \frac{L[F_{\geq n+1}]}{\q{n+1}^2 a_{n+1} } \left( \sum_{m=0}^n a_m q^{n-m} \right).
		\]
		\[
		\frac{q\la}{1-q\la} d_qF =\left( \sum_{n=1}^{\infty} \q{n} a_n \lambda^{n-1} \right)
		\left( \sum_{m=1}^{\infty} q^m \lambda^m \right)
		= \sum_{n=0}^{\infty} \frac{L[F_{\geq n+1}]}{\q{n+1}^2 a_{n+1} } \left( \sum_{m=0}^{n} \q{m} a_m q^{n+1-m} \right).
		\]
		Thus,
		\begin{align*}
		& \quad 2\frac{q^\frac{1}{2}\la-1}{q \la -1}d_qF+\frac{2\q{\tfrac{1}{2}}-1}{q\la-1}F \\
		& =2d_qF+2\frac{q-q^{\frac{1}{2}}}{q} \frac{q\la}{1- q \la }d_qF-\frac{2\q{\tfrac{1}{2}}-1}{1-q\la}F \\
		& =2d_qF+\sum_{n=0}^{\infty} \frac{L[F_{\geq n+1}]}{\q{n+1}^2 a_{n+1} } \sum_{m=0}^{n} \left(2\frac{q-q^{\frac{1}{2}}}{q}\q{m} a_m q^{n+1-m} - (2\q{\tfrac{1}{2}}-1)a_m q^{n-m} \right) \\
		& =2d_qF+\sum_{n=0}^{\infty} \frac{L[F_{\geq n+1}]}{\q{n+1}^2 a_{n+1} } \sum_{m=0}^{n} \left(\left(2(q-q^{\frac{1}{2}})\q{m} - (2\q{\tfrac{1}{2}}-1)\right)a_m q^{n-m} \right) \\
		& =2d_qF+(q-1)\sum_{n=0}^{\infty} L[F_{\geq n+1}]
		\quad (\text{by Lemma}~\ref{aglem} \text{ below}) \\
		& =2\sum_{n=0}^{\infty} \frac{1}{\q{n+1}} L[F_{\geq n+1}]
		+(q-1) \sum_{n=0}^{\infty} L[F_{\geq n+1}] \\
		& = L\left[ \sum_{r=0}^{\infty} a_r \la^r \sum_{n=1}^{r} \left(\frac{2}{\q{n}} +q-1\right)\right].
		\end{align*}
		Hence we obtain
		$
		L\left[F\log_q\la - F\log_q(1-\la)\right] - L\left[ \sum^{\infty}_{n=1}a_n \lambda^n \sum^{n}_{i=1} \left(\frac{2}{[i]_q} +q-1\right)\right]=0.
		$
		
	\end{proof}
	
	\begin{lemm}
		\label{aglem}
		For all $n \in \N$,
		\begin{equation}
		\sum_{m=0}^{n} \left(2(q-q^{\frac{1}{2}})\q{m} - (2\q{\tfrac{1}{2}}-1)\right)a_m q^{n-m}
		=(q-1) \q{n+1}^2 a_{n+1}. \label{eq:aglem}
		\end{equation}
	\end{lemm}
	
	\begin{proof}
		We prove this by induction on $n$.
		If $n=0$, the left-hand side of~\eqref{eq:aglem} is
		$
		-(2\q{\tfrac{1}{2}}-1)=(q-1)\q{\tfrac{1}{2}}^2.
		$
		So it is equal to the right-hand side of~\eqref{eq:aglem}.
		We assume~\eqref{eq:aglem} for $n$. To prove~\eqref{eq:aglem} for $n+1$, it suffices to show
		\begin{equation}
		\label{aglem:induction}
		\left(2(q-q^{\frac{1}{2}})\q{n} -2\q{\tfrac{1}{2}}+1\right)a_n
		=(q-1)\left(\q{n+1}^2a_{n+1}-q\q{n}^2 a_n\right).
		\end{equation}
		We can verify the equation~\eqref{aglem:induction} by simple computations.
	\end{proof}
	
	\begin{rem}
		We have another description of $H$ :
		\[
		H = F \log_q \la +\sum_{n=1}^{\infty} a_n \la^n \left( \sum_{i=1}^{n} \frac{2}{\q{i-\frac{1}{2}}} - \frac{2}{\q{i}} \right).
		\]
		One can show that the right-hand side is annihilated by $L$ in the same way as the proof of $L[H]=0$ in Theorem~\ref{qhyp_solution}.
	\end{rem}
	
	In the rest of this section, we calculate a $q$-analogue of Wronskian
	$W(F,H)=Fd_qH-Hd_qF$.
	in the preparation for the computation in~\S\ref{proofsec}.
	\begin{lemm}
		\label{Wronskian}
		We have
		$
		Fd_qH-Hd_qF=\frac{1}{\la(1-\la)}.
		$
	\end{lemm}
	
	\begin{proof}
		We have
		\begin{align*}
		d_q(\la(1-\la)d_qF)
		& =(1-\q{2}\la)d_qF+\gamma(\la(1-\la))d^2_q F \\
		& =(1-\q{2}\la)d_qF+q\la(1-q\la)d^2_q F \\
		& =(1-\q{2}\la)d_qF-(1-\alpha \la)d_qF +\q{\tfrac{1}{2}}^2 F \\
		& =\q{\tfrac{1}{2}}^2 \gamma(F).
		\end{align*}
		Similarly,
		$
		d_q(\la(1-\la)d_qH)=\q{\tfrac{1}{2}}^2 \gamma(H).
		$
		Thus,
		\begin{align*}
		d_q(\la(1-\la)(Fd_qH-Hd_qF))
		& =d_qF \cdot \la(1-\la)d_qH +\gamma(F) \cdot d_q(\la(1-\la)d_qH) \\
		& \quad -d_qH \cdot \la(1-\la)d_qF -\gamma(H) \cdot d_q(\la(1-\la)d_qF) \\
		& =\gamma(F) \cdot d_q(\la(1-\la)d_qH) -\gamma(H) \cdot d_q(\la(1-\la)d_qF) \\
		& =0.
		\end{align*}
		Since $\la(1-\la)(Fd_qH-Hd_qF) \in Q[[\la]]$ and ${Q[[\la]]}^{d_q=0}=Q$, we see that $\la(1-\la)(Fd_qH-Hd_qF)$ is constant.
		For all $g \in Q[[\la]]$, we have $\la g=0$, $\la d_q(g)=0$	when $\la=0$.
		Therefore, we have
		\begin{align*}
		\la(1-\la)(Fd_qH-Hd_qF) \mid_{\la=0}
		& = \la(1-\la)(Fd_q(F\log_q \la)-Fd_qF \log_q \la ) \mid_{\la=0} \\
		& = \la(1-\la)(Fd_qF \log_q \la+F\gamma(F)\tfrac{1}{\la}-Fd_qF \log_q \la ) \mid_{\la=0} \\
		& = (1-\la)F\gamma(F) \mid_{\la=0} \\
		& =1.
		\end{align*}
	\end{proof}
	
	\section{\texorpdfstring{$q$}{q}-connection and \texorpdfstring{$q$}{q}-hypergeometric differential equation}
	\label{DFandConn}
	Let	$R'$ be one of the rings $R[[\la]]$, $R\langle\la\rangle$, and $R\left\langle\la, \frac{1}{g(\la)}\right\rangle$ $(g(\la)\in R[\la]\backslash(p,q-1)R[\la])$.
	Let $M''$ be the free $R'$ module $R' e_1\oplus R' e_2$ of rank $2$. 
	In this section, we determine a $q$-connection $\nabla_q \colon M'' \to M'' \otimes_{R'} \frac{1}{1-\la} q\Omega_{R'/R}$ (Remark~\ref{rem:q-connection}) which satisfies
	\begin{equation}
	\nabla_q \left(\begin{pmatrix} e_1 & e_2 \end{pmatrix} \begin{pmatrix} f_1 \\ f_2 \end{pmatrix} \right)=0 \Leftrightarrow
	d_q(f_2)=f_1\; \text{and} \; L[f_2]=0, \label{nabla0condition}
	\end{equation}
	which is a $q$-analogue of~\cite[Proposition 7.11 (iii)]{Put} (see the claim before~\eqref{hyp-geom-eq}). Here $f_1$ and $f_2$ are elements of any extension $C$ of $\text{Frac}R'$ which is $(q-1)$-torsion free and carries an extension of $\Gamma$-action satisfying $(\gamma-1)C\subset (q-1)C$.
	Let $P \in \frac{1}{\la(1-\la)}M_2(R')$ and define a $q$-connection $\nabla_q \colon M'' \to M''\otimes_{R'}\frac{1}{1-\la} q\Omega_{R'/R} $ by $ \nabla_q \begin{pmatrix} e_1 & e_2 \end{pmatrix} = \begin{pmatrix} e_1 & e_2 \end{pmatrix} P \otimes d\la$.
	Then,
	\begin{align}
	\nabla_q \left(\begin{pmatrix} e_1 & e_2 \end{pmatrix} \begin{pmatrix} f_1 \\ f_2 \end{pmatrix} \right)=0 
	& \Leftrightarrow \begin{pmatrix} e_1 & e_2 \end{pmatrix} P \gamma \begin{pmatrix} f_1 \\ f_2 \end{pmatrix}
	+ \begin{pmatrix} e_1 & e_2 \end{pmatrix} d_q \begin{pmatrix} f_1 \\ f_2 \end{pmatrix} = 0
	\notag \\
	& \Leftrightarrow (1+(q-1)\la P) \begin{pmatrix} \gamma(f_1) \\ \gamma(f_2) \end{pmatrix}
	= \begin{pmatrix} f_1 \\ f_2 \end{pmatrix}. \label{eq:connP}
	\end{align}
	We define $P'\in M_2(\text{Frac} R')$ by $1+(q-1)\la P=(1+(q-1)\la P')^{-1} $.
	Then the equation~\eqref{eq:connP} is equivalent to $d_q \begin{pmatrix} f_1 \\ f_2 \end{pmatrix}=P' \begin{pmatrix} f_1 \\ f_2 \end{pmatrix}$.
	Hence~\eqref{nabla0condition} holds when
	\[
	P'=\frac{1}{q\la(1-q\la)} \begin{pmatrix} -1+\alpha \la & \q{\tfrac{1}{2}}^2 \\
	q\la(1-q\la) & 0\end{pmatrix}.
	\]
	Therefore,
	\begin{align*}
	1+(q-1)\la P & =(1+(q-1)\la P')^{-1} \\
	& =q\la(1-q\la) \begin{pmatrix} q\la(1-q\la)- (q-1)\la(1-\alpha \la) & (q-1)\la \q{\tfrac{1}{2}}^2 \\ (q-1)\la q\la(1-q\la) & q\la(1-q\la)\end{pmatrix} ^{-1} \\
	& =\frac{1}{\la(1-\la)} \begin{pmatrix}q\la(1-q\la) & -(q-1)\la \q{\tfrac{1}{2}}^2\\ -q\la(1-q\la)(q-1)\la & q\la(1-q\la)- (q-1)\la(1-\alpha \la)\end{pmatrix}
	\end{align*}
	and we obtain
	\[
	P=\frac{1}{\la(1-\la)} \begin{pmatrix} 1-\q{2}\la & -\q{\tfrac{1}{2}}^2 \\ -q\la(1-q\la) & \q{\tfrac{1}{2}}^2 (q-1)\la\end{pmatrix} .
	\]
	
	Let $\overline{e}_1$ and $\overline{e}_2$ be the elements $\la(1-\la)\omega$ and $-\la(1-\la)D(\omega)$ of $H^1_{\text{dR}}$, respectively (see \S\ref{introduction}). Then, in the case $R'=R\left\langle \la,\frac{1}{\la(1-\la)h(\la)}\right\rangle$, the comparison with the formula~\eqref{GMExplicitFormula} shows that we have the following isomorphism of $B'$-modules compatible with the connections 
	\begin{equation} 
	(M'',\nabla_q)\otimes_{R'}B\xrightarrow{\cong}(H^1_{\text{dR}},\nabla); e_1\otimes 1,e_2\otimes 1 \mapsto \overline{e}_1,\overline{e}_2.
	\end{equation}
	
	Put $e_1'=\frac{1}{\la(1-\la)}e_1\in M''\left[\frac{1}{\la(1-\la)}\right]$, and let $M'$ be the free $R'$-submodule $R'e_1'\oplus R'e_2$ of $M''\left[\frac{1}{\la(1-\la)}\right]$. Then the $q$-connection $\nabla_q$ on $M''$ uniquely extends to a $q$-connection on $M''\left[\frac{1}{\la(1-\la)}\right]$, and by a straightforward computation, we see that its restriction to $M'$ gives the $q$-connection $\nabla_q\colon M'\to M'\otimes_{R'}\frac{1}{1-\la}q\Omega_{R'/R}$ on $M'$ defined by \[\nabla_q(e_1',e_2)=(e_1',e_2)\begin{pmatrix}0& -\q{\frac12}^2\\-\frac{1}{\la(1-\la)}& \q{\frac{1}{2}}^2 (q-1)\frac{1}{1-\la} \end{pmatrix}\otimes d\lambda.\]
	We define the filtration on $M'$ by 
	\begin{equation*}
	\Fil^r M'= \begin{cases}
	M' & r=0 \\
	R'e'_1 \oplus (q-1)R' e_2 & r=1.
	\end{cases}
	\end{equation*}
	Then, in the case $R'=R\left\langle \lambda,\frac{1}{h(\la)}\right\rangle$, we have an isomorphism of $B'$-modules with connection and filtration 
	\begin{equation}\label{qHGconnectionModq-1}
	(M',\nabla_q,\Fil^\bullet)\otimes_{R'}B'\xrightarrow{\cong}(H_{B'},\nabla,\Fil^\bullet);e_1',e_2\mapsto \overline{e}_1', \overline{e}_2.
	\end{equation}
	
	\section{\texorpdfstring{$q$}{q}-analogue of \texorpdfstring{$p$}{p}-adic formal congruence}
	\label{Dwsec}
	In this section, we prove some formal congruence between power series in $R[[\lambda]]$ and show that certain constructions give elements of a ring smaller than $\Rla$ by constructing $q$-analogues of some results of ~\cite[\S1-\S3]{Dw}.
	Put 
	\[
	C_\theta(n)=\prod^{n-1}_{\nu=0} \q{\theta+\nu}.
	\]
	Let $\theta \in \Z_p$ be neither zero nor negative rational integer.
	We define $\theta' \in \Q \cap \Z_p$ to be the unique number such that $p \theta'-\theta$ is an ordinary integer in $[0,p-1]$.
	For all $x \in \Q$ , we put
	\begin{equation*}
	\rho (x)=\begin{cases}
	0 & x\leq 0 \\
	1 & x>0.
	\end{cases}
	\end{equation*}
	
	\begin{lemm}~\cite[\S1, Lemma 1 (1,1)]{Dw}
		\label{Dw11}
		Let $\mu,s$ be positive integers. Let $a \in \N \cap [0,p-1]$.
		Then,
		\begin{multline*}
		\frac{C_\theta(a+\mu p+mp^{s+1})}{\varphi(C_{\theta'}(\mu +mp^s))}\equiv \frac{C_\theta(mp^{s+1})}{\varphi(C_{\theta'}(mp^s))} \frac{C_\theta(a+\mu p)}{\varphi(C_{\theta'}(\mu))} \left( \varphi \left(1+q^{\theta'+\mu} \frac{\q{mp^s}}{\q{\theta'+\mu}} \right)\right)^ {\rho(a+\theta-p\theta')}\\
		\mod 1+\q{p^{s+1}}R.
		\end{multline*}
	\end{lemm}
	
	\begin{proof}
		By the definition of $C_{\theta}$,
		\[
		\frac{C_\theta(a+\mu p+mp^{s+1})}{C_{\theta}(mp^{s+1})}
		=
		\prod_{\nu=0}^{a+\nu p-1} \q{\theta+mp^{s+1}+\nu}
		\]
		and
		\[
		\frac{C_\theta(a+\mu p+mp^{s+1})}{C_{\theta}(mp^{s+1}) C_\theta(a+\mu p)}
		=\prod_{\nu=0}^{a+\nu p-1} \frac{\q{\theta+mp^{s+1}+\nu}}{\q{\theta+\nu}}
		=\prod_{\nu=0}^{a+\nu p-1} \left( 1+q^{\theta+\nu} \frac{\q{mp^{s+1}}}{\q{\theta+\nu}}\right).
		\]
		We have
		\[
		\q{mp^{s+1}}=\q{p^{s+1}}\left( 1+q^{p^{s+1}}+q^{p^{2(s+1)}}+\cdots +q^{p^{(m-1)(s+1)}} \right) \in \q{p^{s+1}}R,
		\]
		so the proof of Lemma~\ref{Dw11} is almost the same as the Dwork's proof in~\cite[\S1 Lemma 1 p.31]{Dw}.
	\end{proof}
	
	\begin{lemm}~\cite[\S1, Lemma 1 (1,2)]{Dw}
		\label{Dw12}
		Let $\mu,s$ be positive integers. Then,
		\[
		\frac{C_\theta(mp^{s+1})}{\varphi(C_{\theta'}(mp^s))} \equiv \frac{C_1(mp^{s+1})}{\varphi(C_1(mp^s))} \mod 1+\q{p^{s+1}} R.
		\]
	\end{lemm}
	
	\begin{proof}
		By putting $a=0, \mu=p^s$, the proof of Lemma~\ref{Dw12} is reduced to the case $m=1$.
		For $\nu \in \N \cap [0,p^{s+1}-1]$, the condition
		\[
		\theta +\nu \equiv 0 \mod p
		\]
		is equivalent to the condition that there exists $\nu'\in \N \cap [0,p^{s}-1] $ such that
		$
		\nu=(p\theta'-\theta)+p\nu'.
		$
		This condition implies that $\theta+\nu =p(\theta'+\nu')$.
		Thus, we have
		\[
		\frac{C_\theta(p^{s+1})}{\varphi(C_{\theta'}(p^s))}
		=\prod_{p \mid \theta + \nu} \frac{\q{p(\theta'+\nu')}}{\varphi \left(\q{\theta'+\nu'} \right)}
		\prod_{p \nmid \theta +\nu} \q{\theta+\nu}
		=\q{p}^{p^s} \prod_{\substack{p \nmid \theta +\nu\\ \nu \in \N \cap [0,p^{s+1}-1] }}\q{\theta+\nu}.
		\]
		Especially, by putting $\theta=\theta'=1$, we have (replacing $\nu$ with $\nu_1$)
		\[
		\frac{C_1(p^{s+1})}{\varphi(C_{1}(p^s))}=\q{p}^{p^s} \prod_{\substack{p \nmid 1 +\nu_1 \\ \nu_1 \in \N \cap [0,p^{s+1}-1] }}\q{1+\nu_1}.
		\]
		The sets $\left\lbrace \theta+\nu \mid p \nmid \theta +\nu , \nu \in \N \cap [0,p^{s+1}-1] \right\rbrace $
		and $\left\lbrace 1+\nu_1 \mid p \nmid 1 +\nu_1 , \nu_1 \in \N \cap [0,p^{s+1}-1] \right\rbrace $
		are both representatives of $(\Z / p^{s+1}\Z)^{\times}$, so they have one-to-one correspondence.
		Namely for all $\nu$, there is a unique $\nu_1$ such that
		\[
		\theta+\nu \equiv 1+\nu_1 \mod p^{s+1}.
		\]
		Since $(1+\nu_1,p)=1$, we have $\theta+\nu =(1+\nu_1)(1+p^{s+1}a_{\nu})$ for some $a_\nu \in \Z_p$.
		Thus, we have
		\[
		\frac{C_\theta(p^{s+1})}{\varphi(C_{\theta'}(p^s))}
		=\q{p}^{p^s} \prod_{1+\nu_1} \q{(1+\nu_1)(1+p^{s+1}a_{\nu})}
		=\q{p}^{p^s} \prod_{1+\nu_1} \q{1+\nu_1} \left( 1+q^{1+\nu_1} \frac{\q{p^{s+1}a_{\nu}(1+\nu_1)}}{\q{1+\nu_1}} \right).
		\]
		This is congruent to $\frac{C_1(p^{s+1})}{\varphi(C_{1}(p^s))}\mod 1+\q{p^{s+1}} R$, which follows from $\frac{\q{p^{s+1}a_{\nu}(1+\nu_1)}}{\q{1+\nu_1}} \in \q{p^s+1} R$.
	\end{proof}
	
	Put $A_\theta(n)=C_\theta(n)/ C_1(n) $ for a non-negative number $n$.
	
	\begin{cor}~\cite[\S1, Corollary 2]{Dw}
		Assume that $\theta =\frac{1}{2}$. (Then we have $\theta'=\frac{1}{2}$ by the definition.)
		\begin{enumerate}[(i)]
			\label{Dw13}
			\item \label{Dw13lem}
			Assume $a > \frac{p-1}{2}$. Then
			\[
			\frac{A_\frac{1}{2}(a+\mu p)}{\varphi(A_\frac{1}{2}(\mu))}
			\equiv 0 \mod \q{p^{1+v_p \left(\mu+\frac{1}{2} \right)}} R.
			\]
			\item \label{Dw13prop}
			We have
			\[
			\frac{A_\frac{1}{2}(n+mp^{s+1})}{\varphi \left(A_\frac{1}{2} \left(\left[\frac{n}{p} \right]+mp^s \right) \right)} \equiv \frac{A_\frac{1}{2}(n)}{\varphi \left(A_\frac{1}{2} \left(\left[\frac{n}{p} \right] \right) \right)} \mod \q{p^{s+1}} R.
			\]
		\end{enumerate}
	\end{cor}
	
	\begin{proof}
		\begin{enumerate}[(i)]
			\item Put
			\[
			r(N,m)=\lvert\lbrace\tfrac{1}{2}+i \mid i \in \N \cap [0,N-1], v_p( \tfrac{1}{2}+i ) \geq m \rbrace \rvert -\lvert \lbrace i \mid i \in \N \cap [0,N-1], v_p(i) \geq m \rbrace \rvert.
			\]
			If we write $N=b + p^m c$ ($b,c \in \N, b\in [0,p^m-1]$), we can rewrite $r(N,m)$ by
			\begin{equation*}
			r(N,m)=\begin{cases}
			1 & b\geq \frac{1}{2}(p^m+1) \\
			0 & b < \frac{1}{2}(p^m+1).
			\end{cases}
			\end{equation*}
			Put
			\[
			\Phi_{p^m}(q)=\frac{q^{p^m}-1}{q^{p^{m-1}}-1}.
			\]
			Then modulo units, we have $A_{\frac{1}{2}}(a+\mu p)=\left( \prod_{m=1}^{\infty} \Phi_{p^m}(q)^ {r(a+\mu p,m)} \right)$ , and
			\[
			\varphi\left( A_{\frac{1}{2}}\left( \mu \right) \right)
			= \varphi\left( \prod_{m=1}^{\infty} \Phi_{p^m}(q)^ {r \left( \mu ,m \right)} \right)
			= \prod_{m=1}^{\infty} \Phi_{p^{m+1}}(q)^ {r \left( \mu ,m \right)}
			= \prod_{m=2}^{\infty} \Phi_{p^{m}}(q)^ {r \left(\mu ,m-1 \right)}.
			\]
			Therefore, (by setting $r(N,0)=0$ for $N \in \N$) we obtain
			\[
			\frac{A_{\frac{1}{2}}(a+\mu p)}{\varphi\left( A_{\frac{1}{2}}\left( \mu \right) \right) }
			=\left( \prod_{m=1}^{\infty} \Phi_{p^m}(q)^{ r(a+\mu p,m) - r \left(\mu ,m-1 \right)} \right).
			\]
			So we have to determine $r(a+\mu p,m) - r \left(\mu ,m-1 \right)$ for each $m$.
			
			Put $\mu=\nu +\mu' p^{m-1}$ ($\mu',\nu \in \N, \nu \in [0,p^{m-1}-1]$). Then $r(a+\mu p,m)=0$ means $a+\nu p \leq \frac{1}{2} (p^m-1) $. Thus,
			\[
			\nu =\left[ \frac{a+\nu p}{p} \right] \leq \left[ \frac{p^m-1}{2p} \right] \leq \frac{1}{2}(p^{m-1}-1).
			\]
			So $ r \left(\mu ,m-1 \right)=0$ and this implies $r(a+\mu p,m) - r \left(\mu ,m-1 \right) \geq 0$ for all $m$.
			Assume $m-1 \leq v_p(\mu+\frac{1}{2})$, then we obtain
			\begin{equation}
			v_p \left(\nu+\frac{1}{2} \right) \geq m-1, \label{eq:nu}
			\end{equation}
			because $\nu=\mu-\mu'p^{m-1}$.
			The condition $\nu \in [0,p^{m-1}-1]$ and~\eqref{eq:nu} imply $\nu=\frac{1}{2}(p^{m-1}-1)$. Thus, we obtain
			\begin{gather}
			a+\nu p\geq \frac{p+1}{2}+\frac{p}{2}(p^{m-1}-1) = \frac{1}{2}(p^m+1) \label{1dw13} \\
			\nu \leq \frac{1}{2}(p^{m-1}-1). \label{2dw13}
			\end{gather}
			The inequality~\eqref{1dw13} means $r(a+\mu p,m)=1$ and~\eqref{2dw13} means $r \left(\mu,m-1 \right) = 0$, so we obtain
			$
			r(a+\mu p,m) - r \left(\mu ,m-1 \right) = 1
			$
			under the condition $m \leq 1+v_p(\mu+\frac{1}{2})$.
			Therefore,
			\[\frac{A_{\frac{1}{2}}(n)}{\varphi\left( A_{\frac{1}{2}}\left( \left[ \frac{n}{p}\right] \right) \right) } \in \prod_{m=1}^{1+v_p(\mu+\frac{1}{2})} \Phi_{p^{m}}(q)=\q{p^{1+v_p(\mu+\frac{1}{2})}} .
			\]
			
			\item Put $n=a+\mu p$ ($a,\mu \in \N, a\in [0,p-1]$).
			Then,
			\begin{align*}
			\frac{A_\frac{1}{2}(n+mp^{s+1})}{\varphi \left(A_\frac{1}{2} \left(\left[\frac{n}{p} \right]+mp^s \right) \right)}
			& = \frac{\varphi\left( C_1 ( \mu+mp^s ) \right) }{ C_1 (a+\mu p+mp^{s+1})}
			\frac{C_{\frac{1}{2}} (a+\mu p+mp^{s+1})}{\varphi\left( C_{\frac{1}{2}} ( \mu+mp^s ) \right)} \\
			& \equiv \frac{\varphi\left( C_1 ( \mu+mp^s ) \right) }{ C_1 (a+\mu p+mp^{s+1})}
			\frac{C_{\frac{1}{2}}(mp^{s+1})}{\varphi(C_{\frac{1}{2}}(mp^s))} \frac{C_{\frac{1}{2}}(a+\mu p)}{\varphi(C_{\frac{1}{2}}(\mu))} \\
			& \quad \times \left( \varphi \left(1+q^{{\frac{1}{2}}+\mu} \frac{\q{mp^s}}{\q{{\frac{1}{2}}+\mu}} \right)\right)^ {\rho \left(a-{\frac{p-1}{2}}\right)} \mod \q{p^{s+1}}R
			\quad (\text{by Lemma}~\ref{Dw11}) \\
			& \equiv \frac{\varphi\left( C_1 ( \mu+mp^s ) \right) }{ C_1 (a+\mu p+mp^{s+1})}
			\frac{C_1(mp^{s+1})}{\varphi(C_1(mp^s))} \frac{C_{\frac{1}{2}}(a+\mu p)}{\varphi(C_{\frac{1}{2}}(\mu))} \\
			& \quad \times \left( \varphi \left(1+q^{{\frac{1}{2}}+\mu} \frac{\q{mp^s}}{\q{{\frac{1}{2}}+\mu}}\right)\right)^ {\rho \left(a-{\frac{p-1}{2}}\right)}
			\mod \q{p^{s+1}}R
			\quad (\text{by Lemma}~\ref{Dw12}) \\
			& \equiv \frac{A_{\frac{1}{2}}(a+\mu p)}{\varphi(A_{\frac{1}{2}}(\mu))}
			\left( \varphi \left(1+q^{{\frac{1}{2}}+\mu} \frac{\q{mp^s}}{\q{{\frac{1}{2}}+\mu}} \right)\right)^ {\rho \left(a-{\frac{p-1}{2}}\right)} \mod \q{p^{s+1}}R
			\end{align*}
			by Lemma~\ref{Dw11} again. (Use for $\theta=\theta'=1$.)
			Thus, if $a-{\frac{p-1}{2}} \leq 0$,~\eqref{Dw13prop} is clear. Assume $a-{\frac{p-1}{2}} > 0$, then
			\[
			\frac{A_{\frac{1}{2}}(a+\mu p)}{\varphi(A_{\frac{1}{2}}(\mu))}
			\varphi \left(1+q^{{\frac{1}{2}}+\mu} \frac{\q{mp^s}}{\q{{\frac{1}{2}}+\mu}} \right)-
			\frac{A_{\frac{1}{2}}(a+\mu p)}{\varphi(A_{\frac{1}{2}}(\mu))}
			=\frac{A_{\frac{1}{2}}(a+\mu p)}{\varphi(A_{\frac{1}{2}}(\mu))}
			\varphi \left(q^{{\frac{1}{2}}+\mu} \frac{\q{mp^s}}{\q{{\frac{1}{2}}+\mu}} \right).
			\]
			This is congruent to $0$ modulo $\q{p^{s+1}} R$, which follows from
			\[
			\frac{A_{\frac{1}{2}}(a+\mu p)}{\varphi(A_{\frac{1}{2}}(\mu))} \in \q{p^{1+v_p(\mu+\frac{1}{2})}}R
			\quad \text{by}~\eqref{Dw13lem}
			\]
			and
			\[
			\q{p^{1+v_p(\mu+\frac{1}{2})}} \varphi \left(\frac{\q{p^s}}{\q{ p^{v_p(\mu+\frac{1}{2})} }}\right)
			=\q{p} \varphi \left({\q{p^s}}\right)=\q{p^{s+1}}.
			\]
		\end{enumerate}
	\end{proof}
	
	\begin{theo}~\cite[\S2, Theorem 2]{Dw}
		\label{Dw2}
		Let $A=B^{(-1)} , B=B^{(0)},B^{(1)},B^{(2)},\ldots $, be a sequence of functions on $\N$ with values in $Q\coloneqq \text{Frac}(R)$.
		Put
		\[
		F(\la)=\sum_{n=0}^{\infty} A(n)\la^n,\quad G(\la)=\sum_{n=0}^{\infty} B(n)\la^n.
		\]
		Assume for all $n, m, s \in \N, i\geq -1$,
		\begin{enumerate}[(a)]
			\item $\frac{B^{(i)}(n+mp^{s+1})}{\varphi \left( B^{(i+1)} \left(\left[ \frac{n}{p} \right]+mp^s \right) \right)} \label{dw2hypa}
			\equiv \frac{B^{(i)}(n)}{\varphi \left(B^{(i+1)} \left(\left[\frac{n}{p} \right] \right) \right)} \mod \q{p^{s+1}}R $.
			\item $\frac{B^{(i)}(n)}{\varphi \left(B^{(i+1)} \left(\left[\frac{n}{p} \right] \right) \right)} \in R$.
			\item$B^{(i)}(n) \in R$. \label{dw2hypc}
			\item$B^{(i)}(0)\in R^{\times}$. \label{dw2hypd}
		\end{enumerate}
		Then,
		\[
		F(\la)\varphi \left( \sum_{j=mp^s}^{(m+1)p^s-1} B(j) \la^j \right) \equiv \varphi(G(\la)) \sum_{j=mp^{s+1}}^{(m+1)p^{s+1}-1} A(j) \la^j
		\modphib \Rla.
		\label{maindw2}
		\]
	\end{theo}
	
	\begin{proof}
		Let $n=a+pN$ ($a \in \N \cap [0,p-1]$). So the coefficient of $\la^n$ on the left side of~\eqref{maindw2} is
		\[
		\sum_{j=mp^s}^{(m+1)p^s-1} A(n-pj) \varphi(B(j)),
		\]
		and the coefficient of $\la^n$ on the right side of~\eqref{maindw2} is
		\[
		\sum_{j=mp^s}^{(m+1)p^s-1} \varphi(B(N-j)) A(a+pj) .
		\]
		Let
		\begin{gather*}
		U_a(j,N)=A(a+p(N-j))\varphi(B(j)) - \varphi(B(N-j)) A(a+pj), \\
		H_a(m,s,N)=\sum_{j=mp^s}^{(m+1)p^s-1} U_a(j,N).
		\end{gather*}
		Then what we have to show is
		\begin{equation}
		H_a(m,s,N) \equiv 0 \mod \q{p^{s+1}} \varphi^{s+1}( B^{(s)}(m))R. \label{dw22}
		\end{equation}
		Since $a \in [0,p-1]$, we have $U_a(j,N)=0$ for $j>N$. So
		\begin{equation}
		H_a(m,s,N)=0 \quad \text{for} \quad N<mp^s. \label{dw24}
		\end{equation}
		
		In preparation for the proof of this theorem, we note some facts. The proof of these facts is almost the same as~\cite{Dw}.
		\begin{gather}
		\sum_{m=0}^T H_a(m,s,N)=0 \quad \text{for} \quad (T+1)p^s>N. \label{dw25} \\
		H_a(m,s,N)=\sum_{\mu=0}^{p-1} H_a(\mu+mp,s-1,N) \quad \text{for} \quad s \geq 1. \label{dw26} \\
		B^{(t)}(i+mp^s) \equiv 0 \mod \varphi^s(B^{(s+t)}(m))R
		\quad \text{for} \quad i \in \N \cap [0,p^s-1], s\geq 0, t\geq -1. \label{dw27}
		\end{gather}
		
		We prove~\eqref{dw22} using induction on $s$.
		Put the induction hypothesis
		\[
		(\alpha)_s \colon H_a(m,u,N) \equiv 0 \mod \q{p^{u+1}} \varphi^{u+1}(B^{(u)}(m))R
		\quad \text{for} \quad 0\leq u< s,
		\]
		and the supplementary hypothesis
		\begin{multline*}
		(\beta)_{t,s} \colon H_a(m,s,N+mp^s) \equiv \sum_{j=0}^{p^{s-t}-1}
		\frac{\varphi^{t+1}(B^{(t)}(j+mp^{s-t}) ) H_a(j, t, N)}{\varphi^{t+1}(B^{(t)}(j)) } \\
		\modphib R \quad \text{for} \quad t \in \N \cap [0,s].
		\end{multline*}
		Then $(\alpha)_s$ for all $s\geq 1$ is reduced to the following four claims: \\
		(i) $(\alpha)_1$; (ii) $(\beta)_{0,s}$; (iii) $(\beta)_{t,s}$ and $(\alpha)_s$ imply $(\beta)_{t+1,s}$; (iv) $(\beta)_{s,s}$ implies $(\alpha)_{s+1}$.
		\begin{enumerate}[(i)]
			\item By~\eqref{dw24}, we may assume $N \geq m$.
			By hypothesis~\eqref{dw2hypa},
			\[
			\frac{A(a+p(N-m))}{\varphi(B(N-m))} \equiv \frac{A(a)}{\varphi(B(0))} \mod \q{p}R.
			\]
			Especially if $N=2m$,
			$
			\frac{A(a+pm)}{\varphi(B(m))} \equiv \frac{A(a)}{\varphi(B(0))} \mod \q{p}R.
			$
			Thus,
			\[
			\frac{U_a(m,N)}{\varphi(B(m))\varphi(B(N-m))}
			=\frac{A(a+p(N-m))}{\varphi(B(N-m))}- 	\frac{A(a+pm)}{\varphi(B(m))}
			\equiv 0 \mod \q{p}R.
			\]
			Then by~\eqref{dw2hypc}, $U_a(m,N)\equiv 0 \mod \q{p} \varphi(B(m))R.$
			
			\item We have
			$
			H_a(m,s,N+mp^s)=\sum_{j=0}^{p^s-1}U_a(j+mp^s,N+mp^s)
			$
			and by definition of $U_a$,
			\begin{equation}
			U_a(j+mp^s,N+mp^s)=A(a+p(N-j)) \varphi(B(j+mp^s)) -\varphi(B(N-j)) A(a+pj+mp^{s+1}).
			\label{dw28}
			\end{equation}
			By hypothesis~\eqref{dw2hypa},
			\[
			A(a+pj+mp^{s+1})=\frac{A(a+pj) \varphi(B(j+mp^s))}{\varphi(B(j))}
			+ X_j \varphi(B(j+mp^s))
			\]
			for some $X_j \in \q{p^{s+1}} R$.
			Then the right-hand side of~\eqref{dw28} is
			\begin{align*}
			& \quad A(a+p(N-j)) \varphi(B(j+mp^s))
			-\varphi(B(N-j)) \left(\frac{A(a+pj) \varphi(B(j+mp^s))}{\varphi(B(j))}
			- X_j \varphi(B(j+mp^s))\right) \\
			& =\varphi(B(j+mp^s)) \left(\frac{U_a(j,N)}{\varphi(B(j))}-X_j\varphi(B(N-j)) \right) \\
			& =\varphi(B(j+mp^s)) \left(\frac{H_a(j,0,N)}{\varphi(B(j))}-X_j\varphi(B(N-j)) \right) \\
			& =\frac{\varphi(B(j+mp^s))H_a(j,0,N)}{\varphi(B(j))}-X_j\varphi(B(j+mp^s)) \varphi(B(N-j)). 
			\end{align*}
			We have $\varphi(B(j+mp^s)) \equiv 0 \mod \varphi^{s+1}(B^{(s)}(m))R $ by~\eqref{dw27}.
			So, by combining $X_j \in \q{p^{s+1}} R$, we obtain
			$
			X_j \varphi(B(j+mp^s)) \equiv 0 \modphib R,
			$
			namely
			\[
			U_a(j+mp^s,N+mp^s) \equiv \frac{\varphi(B(j+mp^s))H_a(j,0,N)}{\varphi(B(j))}
			\modphib R.
			\]
			Therefore, we obtain
			\[
			H_a(m,s,N+mp^s)\equiv \sum_{j=0}^{p^s-1} \frac{\varphi(B(j+mp^s))H_a(j,0,N)}{\varphi(B(j))}
			\modphib R,
			\]
			which is $(\beta)_{0,s}$.
			
			\item Put $j=\mu +pi$, then the right-hand side of $(\beta)_{t,s}$ is
			\begin{equation}
			\sum_{\mu=0}^{p-1} \sum_{i=0}^{p^{s-t-1}-1}
			\frac{\varphi^{t+1}(B^{(t)}(\mu +pi+mp^{s-t}) ) H_a(\mu +pi, t, N)}
			{\varphi^{t+1}(B^{(t)}(\mu +pi)) }.
			\label{dwstar}
			\end{equation}
			By hypothesis~\eqref{dw2hypa},
			\[
			B^{(t)}(\mu +pi+mp^{s-t})
			=\frac{B^{(t)}(\mu +pi)\varphi(B^{(t+1)}(i+mp^{s-t-1}))}{\varphi(B^{(t+1)}(i))}
			+X_{i,\mu} \varphi(B^{(t+1)}(i+mp^{s-t-1}))
			\]
			for some $X_{i,\mu} \in \q{p^{s-t}} R$.
			Thus, the general term in the double sum of~\eqref{dwstar} is
			\[
			\frac{\varphi^{t+2}(B^{(t+1)}(i+mp^{s-t-1}) ) H_a(\mu +pi, t, N)}
			{\varphi^{t+2}(B^{(t+1)}(i)) }
			+Y_{i,\mu},
			\]
			where
			\[
			Y_{i,\mu}=\varphi\left( X_{i,\mu} \right) \frac{\varphi^{t+2}(B^{(t+1)}(i+mp^{s-t-1}) ) H_a(\mu +pi, t, N)}
			{\varphi^{t+2}(B^{(t+1)}(\mu+pi)) }.
			\]
			By $(\alpha)_s$, $H_a(\mu +pi, t, N)\equiv 0 \mod \q{p^{t+1}}\varphi^{t+1}(B^{(t)}(\mu+pi))R$, so combining $X_{i,\mu} \in \q{p^{s-t}} R$, we obtain
			\begin{align*}
			Y_{i,\mu} & \equiv 0 \mod \q{p^{s+1}} \varphi^{t+2}(B^{(t+1)}(i+mp^{s-t-1}) )R \\
			& \equiv 0 \modphib R \qquad \text{by~\eqref{dw27}}.
			\end{align*}
			Therefore, the right-hand side of $(\beta)_{t,s} \bmod \q{p^{s+1}} \varphi^{s+1} ( B^{(s)}(m) )R$ is
			\begin{align*}
			& \quad \sum_{\mu=0}^{p-1} \sum_{i=0}^{p^{s-t-1}-1}
			\frac{\varphi^{t+2}(B^{(t+1)}(i+mp^{s-t-1}) ) H_a(\mu +pi, t, N)}
			{\varphi^{t+2}(B^{(t+1)}(i)) } \\
			& =\sum_{i=0}^{p^{s-t-1}-1}
			\frac{\varphi^{t+2}(B^{(t+1)}(i+mp^{s-t-1}) ) H_a(i, t+1, N)}
			{\varphi^{t+2}(B^{(t+1)}(i)) } \qquad \text{by~\eqref{dw26}},
			\end{align*}
			which is $(\beta)_{t+1,s}$.
			\item Let us think about the hypothesis
			\[
			(\gamma)_N \colon H_a(0,s,N)\equiv 0 \mod\q{p^{s+1}}R
			\]
			for $n\in \Z$. We know $(\gamma)_N $ is true for $N\leq 0$.
			Suppose that $\{ N\in \N \mid (\gamma)_N \; \text{fails} \; \} \neq \emptyset$ and put $N' =\min\{ N\in \N \mid (\gamma)_N \; \text{fails} \; \}$.
			Then by $(\beta)_{s,s}$ and hypothesis~\eqref{dw2hypd}, we have
			\[
			\varphi^{s+1}(B^{(s)}(0)) H_a(m,s,N')\equiv \varphi^{s+1}(B^{(s)}(m)) H_a(0,s,N'-mp^s)\\
			\modphib R
			\]
			for $m \geq 1$.
			Since $H_a(0,s,N'-mp^s) \equiv 0 \mod \q{p^{s+1}}R$, we have
			$
			H_a(m,s,N') \equiv 0 \mod \q{p^{s+1}}R
			$
			for $m \geq 1$. By~\eqref{dw25},
			$
			H_a(0,s,N') \equiv 0 \mod \q{p^{s+1}} R,
			$
			which contradicts the definition of $N'$.
			
			Therefore, $(\gamma)_N$ is true for all $N\in\N$, and by $(\beta)_{s,s}$ we obtain
			\[
			H_a(m,s,N+mp^s)\equiv 0 \modphib R
			\]
			for all $N\in\Z $, which implies $(\alpha)_{s+1}$.
		\end{enumerate}
	\end{proof}
	
	\begin{theo}~\cite[\S3, Theorem 3]{Dw}
		\label{Dw3}
		Let $B \colon \N \to R$ be a map satisfying conditions (a), (b), (c) of Theorem~\ref{Dw2} and a condition
		\begin{enumerate}
			\item[(d$'$)] $B(0)=1$.
		\end{enumerate}
		(We set $B^{(i)}=B$ for all $i \geq -1$.)
		Put\[
		F(\la)=\sum_{j=0}^{\infty} B(j)\la^j,\quad F_s(\la)=\sum_{j=0}^{p^s-1} B(j)\la^j.
		\]
		Let $T =R\left<\la,\frac{1}{F_1(\la)}\right>$.
		Then,
		\[
		\frac{F_{s+1}(\la)}{\varphi(F_s(\la))}
		\]
		converges to an element of $T^{\times}$.
	\end{theo}
	
	\begin{proof}
		By Theorem~\ref{Dw2}, for $i \geq 0, s\geq 0$, $F(\la) \varphi(F_s(\la)) \equiv \varphi(F(\la)) F_{s+1}(\la) \mod \q{p^{s+1}} \Rla.$
		Since $F,F_s \in \Rla^{\times}$ by~(d$'$),
		\begin{equation}
		\frac{F_{s+1}(\la)}{\varphi(F_s(\la)) } \equiv \frac{F(\la)}{ \varphi(F(\la))} \mod \q{p^{s+1}} \Rla.
		\label{dw32}
		\end{equation}
		Especially for $s=0$, $F_{1}(\la) \equiv \frac{F(\la)}{ \varphi(F(\la))} \mod \q{p} \Rla.$
		Hence,
		\[
		\frac{F_{s+1}(\la)}{\varphi(F_s(\la)) } \equiv F_{1}(\la) \mod \q{p} \Rla.
		\]
		Therefore, we have
		\begin{equation}
		F_{s+1}(\la) \equiv \varphi(F_s(\la)) F_{1}(\la) \mod \q{p} R[\la],
		\label{dw33}
		\end{equation}
		since the congruence holds $\bmod \q{p} \Rla$ and both sides are polynomials.
		Clearly we have $F_1 \in T ^{\times}$. Since $\q{p} \in (p,q-1) T $ (Proposition~\ref{prop:qanalogue}), we obtain $F_s \in T ^{\times}$ ($s\in \N$) from~\eqref{dw33} by induction on $s$.
		For $s\in \N$, we put
		\[
		f_s(\la)=\frac{F_{s+1}(\la)}{\varphi(F_{s}(\la))}.
		\]
		Then $f_s \in T ^{\times}$.
		By~\eqref{dw32}, we have
		\begin{equation}
		\label{moddw32}
		f_{s+1} \equiv f_{s} \mod {(p,q-1)}^{s+1}\Rla.
		\end{equation}
		Let $R_{s+1}=R/(p,q-1)^{s+1}$.
		Since $F_1(\la) \in {R_{s+1}[[\la]]}^\times$, the natural homomorphism $R_{s+1} [\la]_{F_1(\la)} \to R_{s+1} [[\la]]$ is injective.
		Then by $R_{s+1} [\la]_{F_1(\la)} =R_{s+1}\left[\la, \frac{1}{F_1(\la)}\right]= T /(p,q-1)^{s+1} T,$ the natural homomorphism
		\[
		T /(p,q-1)^{s+1} T \to R_{s+1} [[\la]]=\Rla /(p,q-1)^{s+1} \Rla
		\]
		is injective.
		Therefore by~\eqref{moddw32}, we have $f_{s+1} \equiv f_{s} \mod {(p,q-1)}^{s+1} T.$
		So by the completeness of $ T $, $\{f_s\}_{s\in \N}$ converges to a unit of $ T $.
	\end{proof}
	
	\begin{cor}
		\label{Dwcor}
		Let $F,F_1$ be same as Theorem~\ref{Dw3}.
		Then, $\frac{d_q F}{F}$ and $\frac{d_qF}{\gamma(F)}$ are elements of $ T =R\left\langle \la,\frac{1}{F_1(\la)}\right\rangle $.
	\end{cor}
	
	\begin{proof}
		Put $f=\frac{F}{\varphi(F)}$. By Theorem~\ref{Dw3}, $f, f^{-1}\in T $.
		So we have
		\begin{align*}
		\frac{d_q F}{F} & = \frac{1}{F} d_q(f \varphi(F)) \\
		& = \frac{1}{F} \left(d_q(f) \varphi(F) +\gamma(f) d_q(\varphi(F)) \right) \\
		& = \frac{1}{F} \left(d_q(f) \varphi(F) +\q{p} \la^{p-1} \gamma(f) \varphi(d_q(F)) \right)
		\quad (\text{by Proposition}~\ref{prop:qanalogue}) \\
		& = \frac{d_qf}{f}+\q{p} \la^{p-1}\frac{\gamma(f)}{f} \varphi \left(\frac{d_q F}{F} \right).
		\end{align*}
		By repeating this computation,
		\[
		\frac{d_q F}{F}=\sum_{j=0}^{s-1} \q{p^j} \la^{p^j-1} \left( \prod_{i=0}^{j-1} \varphi^i\left( \frac{\gamma(f)}{f} \right) \right) \varphi^j\left( \frac{d_qf}{f} \right) + \q{p^s} \la^{p^s-1} \left( \prod_{j=0}^{s-1} \varphi^j \left( \frac{\gamma(f)}{f} \right) \right)
		\varphi^s \left( \frac{d_qF}{F} \right).
		\]
		The last term of the right side converges to $0$ in $\Rla$ ($s \to \infty$).
		Put 
		\[
		\eta_s=\sum_{j=0}^{s-1} \q{p^j} \la^{p^j-1} \left( \prod_{i=0}^{j-1} \varphi^i\left( \frac{\gamma(f)}{f} \right) \right) \varphi^j\left( \frac{d_qf}{f} \right).
		\]
		Then by the completeness of $ T $, we obtain $\frac{d_q F}{F}= \lim\limits_{s \to \infty} \eta_s \in T $.
		
		To prove the latter, put $g=f^{-1}=\frac{\varphi(F)}{F}$.
		We have $\frac{d_qF}{\gamma(F)}=-Fd_q \left( \frac{1}{F}\right) $, and
		$
		Fd_q\left(\frac{1}{F}\right)=Fd_q\left(g\frac{1}{\varphi(F) }\right)
		= \frac{d_q g}{g} +\q{p}\la^{p-1} \frac{\gamma(g)}{g} \varphi \left(F d_q \frac{1}{F} \right).
		$
		So discussing similarly, we obtain $\frac{d_qF}{\gamma(F)} \in T $.
	\end{proof}
	
	\section{Proof of the main theorems I}\label{sec:ProofMain}
	\label{proofsec}
	In this section, we prove Theorem~\ref{Main1} and construct $\varphi_{M'}$ of $M'$ satisfying the conditions (i) and (ii) in Theorem~\ref{Main2}.
	
	Let $B(n)=a_n=\prod^{n-1}_{i=0} \left(\frac{[i+\frac{1}{2}]_q}{[i+1]_q} \right)^2\in R.$
	Then by Corollary~\ref{Dw13}, $\{B(n)\}_{n\in \N}$ satisfies the conditions (a), (b), (c) of Theorem~\ref{Dw2} and the condition (d$'$) of Theorem~\ref{Dw3}.
	Therefore, we can apply Theorem~\ref{Dw2}, Theorem~\ref{Dw3}, and Corollary~\ref{Dwcor} to $\{B(n)\}_{n\in \N}$.
	
	Put	$F(\la)= \sum_{n=0}^{\infty} a_n\la^n\in R[[\la]]$ and $F_1(\la)= \sum_{n=0}^{p-1}a_n\lambda^n\in R[\la]$.
	Then, by Theorem~\ref{qhyp_solution}, $F(\la)$ is a solution of the $q$-hypergeometric differential equation~\eqref{eq:qhyp}. Put $S'=R\left\langle \la,\frac{1}{h(\la)}\right\rangle$ as in \S\ref{introduction}. Then~by Lemma~\ref{lem:phi}, we have $S'=R\left\langle \la,\frac{1}{F_1(\la)}\right\rangle =T $.
	
	\begin{proof}[Proof of Theorem~\ref{Main1}]	
		Put $\eta=\frac{d_qF}{F} \in S'$ (Corollary~\ref{Dwcor}), $e'_2=\eta e_1+e_2 \in M'$, and $ U' =S' e'_2\subset M'$. Then by~\eqref{nabla0condition}, $\nabla_q (F e'_2)=\nabla_q(d_qF e_1+F e_2)=0.$
		Put
		\begin{equation}
		\label{nabla2}
		\nabla_q (e'_2)= -\eta' e'_2 \otimes d\la \quad \left(\eta'\in \tfrac{1}{\lambda(1-\la)}S' \right).
		\end{equation}
		Then $\nabla_q (F e'_2)=d_qFe'_2\otimes d\la-\gamma(F)\eta' e'_2 \otimes d\la=0,$ which implies $\eta'=\frac{d_qF}{\gamma(F)}=\eta \frac{F}{\gamma(F)}$.
		(We can determine $\eta'$ directly by the matrix $P$ of \S\ref{DFandConn}.)
		Corollary~\ref{Dwcor} implies $\eta' \in S'$. By Theorem~\ref{Dw3} we have $\frac{F}{\varphi(F)}\in S^{\prime \times}$.
		Define $\varphi_{U'}$ by $\varphi_{U'}(e'_2) =\varepsilon \frac{F}{\varphi(F)} e'_2$, where $\varepsilon=(-1)^{\tfrac{p-1}{2} }$ comes from the classical Frobenius structure~\eqref{unitrootphi} of $U_{B'}$.
		Then the pair $(U', \varphi_{U'})$ satisfies the four conditions in Theorem~\ref{Main1}.
	\end{proof}
	
	\begin{proof}[Proof of Theorems~\ref{Main2} (i), (ii)]
		Put $e_2'=\eta e_1+e_2\in M'$ and $U'=R[[\la]]e_2'\subset M'$ as in the proof of Theorem~\ref{Main1}.
		Let us consider the natural projection $M' \to M' /U'$. Let $\underline{e'_1} \in M'/U' $ be the image of $e'_1$ by this projection. Then,
		\begin{equation}
		\nabla_q(e'_1)=-\frac{1}{\la(1-\la)}e_2 \otimes d\la
		=\left(-\frac{1}{\la(1-\la)}e'_2 +\eta e'_1 \right) \otimes d\la. \label{nabla1}
		\end{equation}
		Therefore, $\nabla_q \left(\underline{e'_1} \right) = \eta \underline{e'_1}\otimes d\la $ and
		\[
		\nabla_q \left(\frac{1}{F}\underline{e'_1} \right)
		=d_q \left(\frac{1}{F}\right) \underline{e'_1}\otimes d\la+ \gamma \left(\frac{1}{F} \right) \nabla_q \left(\underline{e'_1} \right)
		=-\frac{d_q(F)}{F\gamma(F)} \underline{e'_1}\otimes d\la+ \frac{1}{\gamma(F)} \eta \underline{e'_1}\otimes d\la
		=0.
		\]
		Thus, we define a $\varphi$-semilinear endomorphism $\varphi_{M'/U'}$ of $M'/U'$ by $\varphi_{M'/U'}\left(\underline{e'_1}\right) = \varepsilon\q{p} \frac{\varphi(F)}{F} \underline{e'_1}$ so that $\varphi_{M'/U'} \left(\frac{1}{F}\underline{e_1'}\right)$ is a solution of $\nabla_q=0$ and that the reduction modulo $q-1$ of $\varphi_{M'/U'}$ coincides with $\varphi$ of $(H_{B''}/(U_{B'}\otimes_{B'}B''))\otimes_{B''}\Z_p[[\lambda]]$ via~\eqref{Mprime_mod_q-1} (see~\eqref{dRcohphi}). 
		Then its lifting $\varphi_{M'}(e_1')$ should be of the form $\varepsilon \left(\q{p} \frac{\varphi(F)}{F} e'_1 + \q{p}ae'_2 \right)$ for some $a \in \Rla$ by the condition (i) in Theorem~\ref{Main2}.
		
		Let $H$ be the solution of $L=0$ given in Theorem~\ref{qhyp_solution}.
		We have $\nabla_q(d_qH e_1+He_2)=0$ by~\eqref{nabla0condition} and
		\begin{align*}
		d_qH e_1+He_2 & =\la(1-\la) d_qH e'_1+H(e'_2-\la(1-\la)\eta e'_1) \\
		& =\la(1-\la)\frac{Fd_qH-Hd_qF}{F}e'_1+He'_2 \\
		& =\frac{1}{F}e'_1+He'_2.
		\end{align*}
		Here the last equality follows from Lemma~\ref{Wronskian}. Then,
		\begin{align*}
		\varphi_{M'} \left( \begin{pmatrix} e'_1 & e'_2 \end{pmatrix} \begin{pmatrix} 0 &\frac{1}{F} \\ F & H \end{pmatrix} \right) & =
		\begin{pmatrix} e'_1 & e'_2 \end{pmatrix} \varepsilon \begin{pmatrix} \q{p} \tfrac{\varphi(F)}{F} & 0 \\ \q{p}a & \tfrac{F}{\varphi(F)} \end{pmatrix} \begin{pmatrix} 0 &\frac{1}{\varphi(F)} \\\varphi(F) & \varphi(H) \end{pmatrix} \\
		& =\begin{pmatrix} e'_1 & e'_2 \end{pmatrix} \varepsilon\begin{pmatrix}0 &\frac{\q{p}}{F} \\ F &\frac{\q{p}a}{\varphi(F)}+\frac{F\varphi(H)}{\varphi(F)} \end{pmatrix} .
		\end{align*}
		Since $\varphi_{M'} \left(\tfrac{1}{F}e'_1+He'_2\right)$ has to be in the kernel of $\nabla_q$ by the condition (ii) of Theorem~\ref{Main2}, we try to find $a\in R[[\la]]$ such that $\varphi_{M'}\left(\tfrac{1}{F}e'_1+He'_2\right)$ is an $R$-linear combination of $Fe'_2$ and $\tfrac{1}{F}e'_1+He'_2$.
		Suppose that there exists $c \in R$ satisfying
		\[
		\begin{pmatrix} e'_1 & e'_2 \end{pmatrix} \begin{pmatrix}\frac{\q{p}}{F} \\ \frac{\q{p}a}{\varphi(F)}+\frac{F\varphi(H)}{\varphi(F)} \end{pmatrix} =
		\begin{pmatrix} e'_1 & e'_2 \end{pmatrix} \left( \q{p} \begin{pmatrix}\frac{1}{F} \\ H \end{pmatrix} + \q{p}c\begin{pmatrix} 0 \\ F \end{pmatrix} \right),
		\]
		which is equivalent to
		\begin{equation}
		\label{eq:FHac}
		a=\varphi(F)H-\frac{1}{\q{p}}F\varphi(H)+cF\varphi(F) .
		\end{equation}
		We prove
		\begin{equation}
		\label{FHformain}
		\varphi(F)H-\frac{1}{\q{p}}F\varphi(H) \in \Rla.
		\end{equation}
		Put
		\begin{equation*}
		G_1=\sum^{\infty}_{n=1}a_n \lambda^n \sum^{n}_{i=1}\frac{2}{\q{i}} \quad
		G_2=\sum^{\infty}_{n=1}a_n \lambda^n \sum^{n}_{i=1} \left(q-1\right),
		\end{equation*}
		then $
		H=F\log_q\la-F\log_q(1-\la)-G_1-G_2.
		$
		We prove~\eqref{FHformain} by dividing $H$ into these four terms.
		\begin{enumerate}[(i)]
			\item
			$
			\varphi(F)F\log_q \la-\frac{1}{\q{p}}F\varphi(F\log_q\la)
			= \varphi(F)F\log_q \la-\frac{1}{\q{p}}F\varphi(F)\varphi(\log_q\la)
			= 0.
			$
			\item \begin{align*}
			-\varphi(F)F\log_q (1-\la)+\frac{1}{\q{p}}F\varphi(F\log_q(1-\la))
			& = F\varphi(F) \left(\sum_{n=1}^{\infty} \frac{\la^n}{\q{n}}- \frac{1}{\q{p}} \sum_{n=1}^{\infty} \frac{\la^{pn}}{\varphi (\q{n})}\right) \\
			& =F\varphi(F) \sum_{(n,p)=1} \frac{\la^n}{\q{n}}.
			\end{align*}
			This is an element of $\Rla$.
			\item \begin{align*}
			& \quad -\varphi(F)G_1+\frac{1}{\q{p}}F\varphi(G_1) \\
			& =-\varphi(F)\sum^{\infty}_{n=1}a_n \lambda^n \sum^{n}_{i'=1} \frac{2}{\q{i'}}
			+\frac{1}{\q{p}}F\varphi \left(\sum^{\infty}_{n=1}a_n \lambda^n \sum^{n}_{i'=1} \frac{2}{\q{i'}} \right) \\
			& =-2\varphi(F) \sum^{\infty}_{i=1} \frac{1}{\q{i}}\sum^{\infty}_{n=i}a_n \lambda^n+2\frac{1}{\q{p}}F\varphi \left( \sum^{\infty}_{i'=1} \frac{1}{\q{i'}} \left(\sum^{\infty}_{n=i'}a_n \lambda^n\right) \right) \\
			& =-2\varphi(F) \sum_{(i,p)=1} \frac{1}{\q{i}}\sum^{\infty}_{n=i}a_n \lambda^n
			-2\varphi(F) \sum_{p \mid i} \frac{1}{\q{i}}\sum^{\infty}_{n=i}a_n \lambda^n
			+2F \sum^{\infty}_{i'=1} \frac{1}{\q{pi'}} \varphi \left( \sum^{\infty}_{n=i'}a_n \lambda^n\right).
			\end{align*}
			The first term is an element of $\Rla$, and the remaining two terms are equal to
			\begin{align*}
			2 \sum^{\infty}_{i'=1} \frac{1}{\q{pi'}}
			\left(-\varphi(F)\sum^{\infty}_{n=pi'}a_n \lambda^n+F\varphi \left( \sum^{\infty}_{n=i'}a_n \lambda^n \right) \right).
			\end{align*}
			By Theorem~\ref{Dw2},
			\[
			-\varphi(F)\sum^{pi'-1}_{n=0}a_n \lambda^n+F\varphi \left( \sum^{i'-1}_{n=0}a_n \lambda^n \right) \in \q{p^{v_p(i')+1}} R[[\la]]=\q{pi'}R[[\la]],
			\]
			which implies
			$
			-\varphi(F)\sum^{\infty}_{n=pi'}a_n \lambda^n+F\varphi \left( \sum^{\infty}_{n=i'}a_n \lambda^n \right) \in \q{pi'}R[[\la]].
			$
			Therefore, we have
			$
			-\varphi(F)G_1+\frac{1}{\q{p}}F\varphi(G_1) \in \Rla.
			$
			\item
			$ -\varphi(F)G_2+\frac{1}{\q{p}}F\varphi(G_2) =-(q-1)\varphi(F)\sum^{\infty}_{n=1} n a_n \lambda^n +(q-1) F \varphi \left(\sum^{\infty}_{n=1}n a_n \lambda^n \right) \in \Rla.$
		\end{enumerate}
		By adding all of them, we obtain $-\varphi(F)H+\frac{1}{\q{p}}F\varphi(H) \in \Rla.$
		
		We define $\varphi_{M'}$ by choosing $c\in R$ and using $a\in \Rla$ defined by~\eqref{eq:FHac}; we set $\varphi_{M'}(e_2')=\varepsilon\frac{F}{\varphi(F)}e_2'$ as in the proof of Theorem~\ref{Main1}.
		Then, $\varphi_{M'}$ satisfies the condition~\eqref{Main2:cond1} of Theorem~\ref{Main2}.
		We show the condition~\eqref{Main2:cond2} of Theorem~\ref{Main2}.
		We have
		\begin{equation}
		\label{phirep}
		\begin{split}
		\gamma_{M'}(e'_1)&=\frac{\gamma(F)}{F}e'_1-(q-1)\frac{1}{1-\la}e'_2 \qquad \text{by}~\eqref{nabla1} \\
		\gamma_{M'}(e'_2)&=\frac{F}{\gamma(F)}e'_2 \qquad \text{by}~\eqref{nabla2} \\
		\varphi_{M'}(e'_1)&= \varepsilon\q{p} \frac{\varphi(F)}{F} e'_1 + \varepsilon\q{p}ae'_2\\
		\varphi_{M'}(e'_2)&= \varepsilon\frac{F}{\varphi(F)} e'_2.
		\end{split}
		\end{equation}
		The equation $\varphi_{M'}(\gamma_{M'}(e'_2))=\gamma_{M'} (\varphi_{M'}(e'_2))$ follows from $\varphi \circ \gamma=\gamma \circ \varphi$ on $Q((\la))$.
		By calculating $\varphi_{M'}(\gamma_{M'}(e'_1))$ and $\gamma_{M'} (\varphi_{M'}(e'_1))$, the equality $\varphi_{M'} \circ \gamma_{M'}(e'_1)=\gamma_{M'} \circ \varphi_{M'}(e'_1)$ holds if and only if
		\begin{align*}
		& \q{p}a\frac{\varphi(\gamma(F))}{\varphi(F)}
		-(q-1)\q{p} \frac{1}{\varphi(1-\la)}\frac{F}{\varphi(F)}
		=-(q-1)\q{p}\frac{1}{1-\la}\frac{\gamma(\varphi(F))}{\gamma(F)}
		+\q{p}\gamma(a) \frac{F}{\gamma(F)} \\
		\Leftrightarrow & \la d_q \left( \frac{a}{F\varphi(F)} \right)
		=\frac{1}{1-\la} \frac{1}{F\gamma(F)}
		-\varphi \left( \frac{1}{1-\la} \frac{1}{F\gamma(F)} \right).
		\end{align*}
		On the other hand, by~\eqref{eq:FHac}, we have
		$
		\frac{a}{F\varphi(F)} =\frac{H}{F}-\frac{1}{\q{p}} \varphi \left( \frac{H}{F}\right)+c.
		$
		We calculate $d_q \left( \frac{H}{F}\right)$ and $d_q\left( \frac{1}{\q{p}} \varphi\left(\frac{H}{F}\right)\right)$.
		Note that $d_q c=0$ by $c\in R$. We have
		\[
		d_q \left( \frac{H}{F}\right)=d_qH \frac{1}{\gamma(F)}+ Hd_q\left( \frac{1}{F}\right)
		=\frac{Fd_qH-Hd_qF}{F\gamma(F)}
		=\frac{1}{\la(1-\la)F\gamma(F)} \quad (\text{by Lemma}~\ref{Wronskian}),
		\]
		\[
		d_q\left( \frac{1}{\q{p}} \varphi\left(\frac{H}{F}\right)\right)
		=\frac{1}{\q{p}} d_q\left( \varphi\left(\frac{H}{F}\right)\right)
		=\frac{1}{\q{p}} \q{p}\la^{p-1} \varphi\left( d_q\left(\frac{H}{F}\right)\right)
		=\la^{p-1} \varphi\left( \frac{1}{\la(1-\la)F\gamma(F)}\right).
		\]
		Thus, we obtain
		\[
		\la d_q \left( \frac{a}{F\varphi(F)} \right)=
		\la d_q\left( \frac{H}{F}-\frac{1}{\q{p}} \varphi \left( \frac{H}{F}\right)+c \right)
		=\frac{1}{(1-\la)F\gamma(F)} - \varphi\left( \frac{1}{(1-\la)F\gamma(F)}\right)
		\]
		Therefore, we have
		$
		\varphi_{M'} \circ \gamma_{M'}(e'_1)=\gamma_{M'} \circ \varphi_{M'}(e'_1).
		$
		
		In conclusion, if we choose $c \in R$ and define $a$ by~\eqref{eq:FHac}.
		Then $\varphi_{M'}$ defined by~\eqref{phirep}
		satisfies the conditions (i) and (ii) of Theorem~\ref{Main2}.
	\end{proof}
	
	\section{Proof of the main theorems II}
	\label{sec:answer}
	Assume that $R'=\Rla$.
	In this section, we prove the condition (iii) in Theorem~\ref{Main2} holds for $\varphi_{M'}$ constructed in \S\ref{sec:ProofMain} for a suitable $c\in R$.
	First, by taking the image of $\MICHB$ under the functor~\eqref{intro:ARJ} with $a=1$,
	we obtain an $R/(q-1)^2 [[\la]]$-module $ H_{0} \otimes R/(q-1)^2 [[\la]]$ with a filtration, a Frobenius endomorphism and a $\Gamma$-action.
	Second, by taking the image of $(M', \Fil^{\bullet} M', \varphi_{M'}, \rho_{M'})$ in $\text{MF}_{[0,a]}^{\q{p}, q-1}(R'/(q-1)^2,\varphi,\Gamma)$ under the equivalence of categories~\eqref{intro:RJ2} with $a=1$,
	we obtain an $R/(q-1)^2 [[\la]]$-module $M'/(q-1)^2 M'$ with a filtration, a Frobenius endomorphism, and a $\Gamma$-action.
	
	By the construction of the canonical $q$-deformation of $\MICHB$, it suffices to show that there exists an isomorphism $g \colon M'/(q-1)^2M' \to H_0 \otimes R/(q-1)^2 [[\la]]$ in $\text{MF}_{[0,1]}^{\q{p}, q-1}(R'/(q-1)^2,\varphi,\Gamma)$ for a suitable choice of $c \in R$ such that $g$ mod $q-1$ coincides with the isomorphism~\eqref{Mprime_mod_q-1}.
	Let $B_1 \in M_2(\Z_p[[\la]])$, and define an $R/(q-1)^2[[\la]]$-linear lifting $g$ of~\eqref{Mprime_mod_q-1} by
	\[
	g \begin{pmatrix} e'_1 & e'_2 \end{pmatrix} =
	\begin{pmatrix} \overline{e}'_1 \otimes 1 & \overline{e}'_2 \otimes 1 \end{pmatrix}
	(1+(q-1)B_1).
	\]
	It is clear that $g$ is a filtered isomorphism.
	Since $\gamma_{M'}=1+(q-1)D_q^{\log}$, the compatibility of $g$ with the $\Gamma$-actions is equivalent to that of the compatibility with the connections $\bmod$ $q-1$.
	The latter is clear by~\eqref{Mprime_mod_q-1}.
	
	Set $\varphi_{M'/(q-1)^2M'} \begin{pmatrix} e'_1 & e'_2 \end{pmatrix}=\begin{pmatrix} e'_1 & e'_2 \end{pmatrix} (A_0+(q-1) A_1)$ for $A_0, A_1 \in M_2(\Z_p[[\la]])$.
	We choose $c\in R$ such that $a$ defined by~\eqref{eq:FHac} is the lift of $b$ in~\eqref{dRcohphi}; we can show that such a $c$ exists by taking the reduction modulo $q-1$ of the proof of Theorem 4.2 (i), (ii) in \S\ref{sec:ProofMain}.
	Then we have
	\[
	\varphi_{H_{0}}{\otimes\varphi} \begin{pmatrix} \overline{e}'_1 \otimes 1 & \overline{e}'_2 \otimes 1 \end{pmatrix}=\begin{pmatrix} \overline{e}'_1 \otimes 1 & \overline{e}'_2 \otimes 1 \end{pmatrix} A_0.
	\]
	We determine $B_1 \in M_2(\Z_p[[\la]])$ satisfying $\varphi_{H_{0}} \circ g =g \circ \varphi_{M'/(q-1)^2M'}$.
	We have
	\begin{align*}
	(\varphi_{H_{0}}\otimes\varphi)\circ g \begin{pmatrix} e'_1 & e'_2 \end{pmatrix}
	= & \begin{pmatrix} \overline{e}'_1 \otimes 1 & \overline{e}'_2 \otimes 1 \end{pmatrix}
	A_0 \varphi(1+(q-1)B_1), \\
	g \circ \varphi_{M'/(q-1)^2M'} \begin{pmatrix} e'_1 & e'_2 \end{pmatrix}
	= & \begin{pmatrix} \overline{e}'_1 \otimes 1 & \overline{e}'_2 \otimes 1 \end{pmatrix}
	(1+(q-1)B_1)(A_0+(q-1)A_1).
	\end{align*}
	Therefore the compatibility of $g$ with Frobenius is equivalent to
	\begin{align*}
	\quad & A_0 \varphi(1+(q-1)B_1)=(1+(q-1)B_1)(A_0+(q-1)A_1) \\
	\Leftrightarrow & A_0+(q^p-1)A_0\varphi(B_1)=A_0+(q-1)(A_1+B_1 A_0) \quad \text{by } (q-1)^2 \equiv 0 \\
	\Leftrightarrow & \q{p} A_0\varphi(B_1)=A_1+B_1 A_0 \qquad \text{in } M_2(R/(q-1)R[[\la]]) \\
	\Leftrightarrow & B_1- pA_0\varphi(B_1)A_0^{-1} =-A_1A_0^{-1} \qquad \text{in } M_2(\Z_p[[\la]]) .
	\end{align*}
	Let us consider the $\varphi$-semilinear map $\mathcal{F} \colon M_2(\Z_p[[\la]]) \to M_2(\Z_p[[\la]])$ defined by $\mathcal{F}(X)=pA_0\varphi(X)A_0^{-1}$.
	(Since $\det A_0=p^{-1}$, $pA_0\varphi(X)A_0^{-1}$ is an element of $M_2(\Z_p[[\la]])$.)
	
	Then what we want to show is that we can choose $c\in R$ so that $(a\bmod q-1)=b$ and
	\begin{equation}
	\label{eq:B_1}
	\text{there exists $B_1 \in M_2(\Z_p[[\la]])$ satisfying $(1-\mathcal{F})(B_1)=-A_1A_0^{-1}$}.
	\end{equation}
	Note that $A_1A_0^{-1} \in M_2(\Z_p[[\la]])$ because
	$
	(A_0+(q-1) A_1){A_0}^{-1}=1+(q-1)A_1A_0^{-1}
	$
	and, by letting $f= F\mid_{q=1}$, we have
	\begin{align*}
	(A_0+(q-1) A_1){A_0}^{-1}
	& \equiv\begin{pmatrix} \q{p} \frac{\varphi(F)}{F} & 0 \\ \q{p}a & \frac{F}{\varphi(F)} \end{pmatrix}
	\begin{pmatrix} \frac{f}{p\varphi(f)} & 0 \\ -b & \frac{\varphi(f)}{f} \end{pmatrix}
	\mod (q-1)^2 M_2(\Rla) \\
	& \equiv \begin{pmatrix} \frac{\q{p}\varphi(F)f}{pF\varphi(f)} & 0 \\
	a\frac{\q{p}f}{p\varphi(f)}-b\frac{F}{\varphi(F)} & \frac{F\varphi(f)}{\varphi(F)f}\end{pmatrix} \mod (q-1)^2 M_2(\Rla),
	\end{align*}
	which is an element of $M_2(R/(q-1)^2[[\la]])$ by $\q{p}=p\cdot \text{unit}$ in $R/(q-1)^2$.
	
	We first prove the claim modulo $\lambda$. Let $A_1(0)$, $A_2(1)\in R$ be the values of $A_1$ and $A_2$ at $\lambda=0$.
	\begin{lemm}
		\label{FrobEquation}
		We can choose $c\in R$ so that $(a \bmod q-1)=b$ and there exists $ C \in M_2(\Z_p)$ satisfying $C-pA_0(0)CA_0(0)^{-1}=-A_1(0)A_0(0)^{-1}$.
	\end{lemm}
	
	\begin{proof} 
		Choose $c\in R$ such that $(a\bmod q-1)=b$.
		By \S\ref{proofsec}, we have
		\[
		A_0+(q-1)A_1 \equiv \begin{pmatrix} \q{p} \frac{\varphi(F)}{F} & 0 \\ \q{p}a & \frac{F}{\varphi(F)} \end{pmatrix} \mod (q-1)^2 R.
		\]
		Let $b_0, b_1 \in \Z_p$ satisfy $a\mid_{\la=0} = b_0+(q-1)b_1 \mod (q-1)^2 R$. (We have $b \mid_{\la=0}=b_0$.)
		Then, we have
		\[
		A_0(0)=\begin{pmatrix} p & 0 \\ pb_0 & 1 \end{pmatrix}
		\quad
		A_1(0)=\begin{pmatrix} \frac{p(p-1)}{2} & 0 \\ pb_1 & 0 \end{pmatrix}
		\]
		because $F(0)=1$ and $\q{p}=p+(q-1)\binom{p}{2}+(q-1)^2\binom{p}{3}+\cdots$.
		Put $C=\begin{pmatrix} x&y\\z&w \end{pmatrix}.$
		Then,
		\begin{align*}
		C- pA_0(0)CA_0(0)^{-1}
		& =\begin{pmatrix} (1-p)x+p^2 b_0 y& (1-p^2) y \\ -pb_0 x+p^2{b_0}^2y+pb_0 w & -p^2b_0y+(1-p)w \end{pmatrix},
		\end{align*}
		\[
		-A_1(0)A_0(0)^{-1}=
		-\begin{pmatrix} \frac{1}{2}p(p-1) & 0 \\ pb_1 & 0 \end{pmatrix}
		\begin{pmatrix} \frac{1}{p}&0\\-b_0& 1 \end{pmatrix}
		=-\begin{pmatrix} \frac{1}{2}(p-1) & 0 \\ b_1 & 0 \end{pmatrix}.
		\]
		Hence the equality in the lemma holds if and only if
		\begin{gather}
		(1-p)x+p^2b_0y=-\frac{1}{2} (p-1) \label{B_1:11} \\
		(1-p^2) y =0 \label{B_1:12} \\
		-pb_0 x+p^2{b_0}^2y+pb_0 w=-b_1 \label{B_1:21} \\
		-p^2b_0y+(1-p)w=0. \label{B_1:22}
		\end{gather}
		The equations~\eqref{B_1:11},~\eqref{B_1:12}, and~\eqref{B_1:22} are equivalent to $x=\frac{1}{2}, y=w=0$.
		This solution satisfies the equation~\eqref{B_1:21} if and only if $b_1=\frac{1}{2}p b_0.$
		Since $F\varphi(F)\in 1+(q-1)R[[\la]]$, one can replace $c$ by $c+(q-1)c'$ for some $c'\in R[[\lambda]]$ so that this equality holds.	
	\end{proof}
	
	Choose $c\in R$ satisfying the condition in Lemma~\ref{FrobEquation}. We show that~\eqref{eq:B_1} holds for $\varphi_{M'}$ defined by $c$.
	We have $\mathcal{F}^n(\la M_2(\Z_p[[\la]])) \subset \la^{p^{n}} M_2(\Z_p[[\la]])$ for all $n \in \N$. Thus, $\sum_{n=0}^{\infty} \mathcal{F}^n$ converges to an endomorphism of $\la M_2(\Z_p[[\la]])$.
	Therefore, $1-\mathcal{F}$ is bijective on $\la M_2(\Z_p[[\la]])$,
	because $\sum_{n=0}^{\infty} \mathcal{F}^n$ is the inverse of $1-\mathcal{F}$.
	
	We have the following commutative diagram whose two horizontal lines are exact.
	\[
	\xymatrix{
		0  \ar[r] &\la M_2(\Z_p[[\la]])  \ar[d]^{1-\mathcal{F}}_{\cong} \ar[r]^-{i_1} &M_2(\Z_p[[\la]])  \ar[d]^{1-\mathcal{F}} \ar[r]^-{\pi_1} &M_2(\Z_p) \ar[d]^{1-\mathcal{F}(0)} \ar[r] & 0 \\
		0 \ar[r] &\la M_2(\Z_p[[\la]])  \ar[r]^-{i_2} &M_2(\Z_p[[\la]]) \ar[r]^-{\pi_2}&M_2(\Z_p) \ar[r] & 0
	}
	\]
	By the choice of $c$, there exists $C \in M_2(\Z_p)$ satisfying $(1-\mathcal{F}(0))(C)=-A_1(0)A_0(0)^{-1}$.
	Then by the surjectivity of $\pi_1$, there exists $\tilde{C} \in M_2(\Z_p[[\la]]) $ such that
	$\pi_1(\tilde{C})=C$.
	By the commutativity of the right square, $\pi_2(1-\mathcal{F})(\tilde{C})=-A_1(0)A_0(0)^{-1}$,
	and therefore $-A_1A^{-1}_0-(1-\mathcal{F})(\tilde{C})$ lies in the kernel of $\pi_2$.
	By the exactness of the lower horizontal line, there exists $D \in \la M_2(\Z_p[[\la]]) $
	which satisfies $i_2(D)=-A_1A^{-1}_0-(1-\mathcal{F})(\tilde{C})$.
	Put $E=(1-\mathcal{F})^{-1}(D)$ in $\la M_2(\Z_p[[\la]]) $.
	Then
	$
	(1-\mathcal{F}) (i_1(E)+\tilde{C})=-A_1A^{-1}_0-(1-\mathcal{F})(\tilde{C})+(1-\mathcal{F})(\tilde{C})=-A_1A^{-1}_0.
	$
	This completes the proof.
	
	\section{A further topic}
	\label{FurtherTopic}
	Let $R'$ be one of the rings $S'$ and $R[[\la]]$.
	The $\Gamma$-action on $R'$ is geometric in the sense that it defines through the coordinate $\la$ relevant to $q$-connection.
	In this section, we introduce an arithmetic action on $R'$ via the coefficient ring $R$ and show that the unit root part $U'$ of $M'$ admits an arithmetic action.
	
	Let $U'$ be the unit root part of $M'$ given in Theorem~\ref{Main1}.
	For $l \in \Z_p^ {\times}$, we define an automorphism $\sigma_l$ of $S'$, and also of $R[[\la]]$, by $\sigma_l(a)=a$ ($a\in \Z_p$), $\sigma_l(q)=q^l$, and $\sigma_l(\la)=\la$.
	This $\sigma_l$ satisfies $\sigma_l \circ \varphi =\varphi \circ \sigma_l$.
	Let $F\in R[[\la]]$ be the solution of the $q$-differential equation~\eqref{eq:qhyp} given in Theorem~\ref{qhyp_solution}.
	\begin{lemm}
		\label{qHGSolutionArithAction}
		$\frac{F}{\sigma_l(F)}$ is an element of $S^{\prime\times}$.
	\end{lemm}
	
	\begin{proof}
		For each $n$, put $a'_n=a_n \mid_{q=1} \in \Z_p$.
		Then there is a unique $a''_n \in R$ such that $a_n=a'_n+(q-1)a''_n$.
		For $r \geq 0$, we have
		\begin{align*}
		\varphi^{r+1}(F)=\sum_{n=0}^{\infty} \varphi^{r+1}(a_n) \la^{p^{r+1} n}
		& =\sum_{n=0}^{\infty} (a'_n+ \varphi^{r+1}(q-1)(a''_n)) \la^{p^{r+1} n} \\
		& \equiv \sum_{n=0}^{\infty} a'_n \la^{p^{r+1} n} \mod \q{p^{r+1}},
		\end{align*}
		and similarly we have $\sigma_l( \varphi^{r+1}(F))\equiv \sum_{n=0}^{\infty} a'_n \la^{p^{r+1} n} \mod \q{p^{r+1}}$.
		Therefore,
		$
		\varphi^{r+1}(F) \equiv \sigma_l( \varphi^{r+1}(F)) \mod {(p,q-1)}^{r+1}.
		$
		Since $\sigma_l( \varphi^{r+1}(F)) $ is a unit of $\Rla$, we have
		\begin{equation}
		\frac{\varphi^{r+1}(F)}{\sigma_l( \varphi^{r+1}(F))} \equiv 1\mod {(p,q-1)}^{r+1}.
		\end{equation}
		Put $f=\frac{F}{\varphi(F)}\in S^{\prime\times}$. Then we have $\frac{F}{\sigma_l(F)}=\frac{f\varphi(f)\cdots\varphi^r(f)}{\sigma_l(f\varphi(f)\cdots\varphi^r(f))}\cdot\frac{\varphi^{r+1}(F)}{\sigma_l(\varphi^{r+1}(F))}$, and the first term of the right-hand side is contained in $S^{\prime\times}$. Hence,
		the same argument as the proof of Theorem~\ref{Dw3} shows that $\frac{F}{\sigma_l(F)}$ is an element of $S^{\prime\times}$.
	\end{proof}
	
	Let $\widehat{\Gamma}$ be the inverse limit $\varprojlim_n\Gamma/\Gamma^{p^n}$, which is isomorphic to $\Z_p$. For $R'=S', R[[\la]]$, the triviality modulo $q-1$ of the action of $\Gamma$ on $R'$ implies that the action is continuous with respect to the $p$-adic topology of $\Gamma$ and the $(p,q-1)$-adic topology of $R'$. Therefore the action of $\Gamma$ on $R'$ uniquely extends to a continuous action $\widehat{\rho}\colon \widehat{\Gamma}\to \text{Aut}(R')$ of $\widehat{\Gamma}$ on $R'$.
	We have $\widehat{\rho}(\gamma^m)(a)=a$ $(a\in R)$ and $\widehat{\rho}(\gamma^m)(\la) =q^m\la$ for $m\in \Z_p$.
	Similarly, the action $\rho_{U'}$ of $\Gamma$ on $U'$ uniquely extends to a continuous $\widehat{\rho}$-linear action $\widehat{\rho}_{U'}\colon \widehat{\Gamma}\to \text{Aut}(U')$ of $\widehat{\Gamma}$ on $U'$.
	The formula $\nabla_q(e_2')=-\eta\frac{F}{\gamma(F)}e_2' \otimes d\lambda$ shown in the proof of Theorem~\ref{Main1} implies $\rho_{U'}(e_2')=\frac{F}{\gamma(F)}e_2'$. Hence we have
	\[
	\widehat{\rho}_{U'}(\gamma^m)(e'_2)=\frac{F}{\gamma^m(F)}e'_2
	\]
	for $m\in \Z_p$.
			
	Let $\Gamma'$ be the group $\Z_p^{\times}$.
	We define a homomorphism $\rho' \colon \Gamma' \to \text{Aut}(S')$ trivial modulo $q-1$ by $\rho'(l)(s)=\sigma_l(s)$.
	By Lemma~\ref{qHGSolutionArithAction}, we can define a $\sigma_l$-semilinear automorphism $\sigma_{l,U'}$ of $U'$ by
	\[
	\sigma_{l,U'}(e'_2)=\frac{F}{\sigma_l(F)}e'_2.
	\]
	We define the $\rho'$-semilinear action $ \rho'_{U'} \colon \Gamma' \to \text{Aut}(U')$ of $\Gamma'$ on $U'$ by $\rho'_{U'}(l)(u)=\sigma_{l,U'}(u)$.
	
	Let $\Gamma'\ltimes\widehat{\Gamma}$ be the semi-direct product defined by the canonical action of $\Gamma'=\Z_p^{\times}$ on $\widehat{\Gamma}\cong\Z_p$.
	Since $\sigma_l\widehat{\rho}(\gamma^m)\sigma_l^{-1}=\widehat{\rho}(\gamma^{ml})$ on $S'$ and on $R[[\la]]$ for $l\in \Z_p^{\times}$ and $m\in \Z_p$, we can define an action $\rho' \ast \widehat{\rho} \colon \Gamma'\ltimes\widehat{\Gamma} \to \text{Aut}(S')$ by $\widehat{\rho}$ and $ \rho'$,
	and a $\rho' \ast \widehat{\rho}$-semilinear action $\left(\rho' \ast \widehat{\rho} \right)_{U'} \colon \Gamma'\ltimes\widehat{\Gamma} \to \text{Aut}(U')$ by $\widehat{\rho}_{U'}$ and $\rho'_{U'}$.
	
	By using $\varphi\circ \sigma_l=\sigma_l \circ \varphi$ on $R[[\la]]$, we see that the triplet $\left( U',\varphi_{U'},\left(\rho' \ast \widehat{\rho}\right)_{U'} \right)$ is an object of the category $\text{MF}_{[0,0]}^{\q{p}, q-1} \left(S', \varphi, \Gamma'\ltimes\widehat\Gamma\right)$, whose image under the equivalence of categories (\cite[Proposition 56]{Tsu})
	\[
	\text{MF}_{[0,0]}^{\q{p},q-1}\left(S',\varphi,\Gamma'\ltimes \widehat{\Gamma}\right)
	\xrightarrow{\bmod q-1} 
	\text{MF}_{[0,0]}(B',\varphi)
	\]
	is $\left( U_{B'},\varphi_{U_{B'}} \right)$.

\end{document}